\documentclass[10pt]{article}

\usepackage{amsmath}
\usepackage{amsfonts}
\usepackage{amsbsy}
\usepackage{amssymb}
\usepackage{graphicx}
\usepackage{psfrag}
\usepackage{epsfig}
\usepackage{multicol}

\def\beq{\begin{equation}}  \def\eeq{\end{equation}}
\def\beqn{\begin{eqnarray}} \def\eeqn{\end{eqnarray}}
\def\beqnn{\begin{eqnarray*}}   \def\eeqnn{\end{eqnarray*}}
\def\barr{\begin{array}}    \def\earr{\end{array}}
\def\bit{\begin{itemize}}   \def\eit{\end{itemize}}
\def\ben{\begin{enumerate}} \def\een{\end{enumerate}}
\def\bc{\begin{center}}     \def\ec{\end{center}}
\def\eps{\varepsilon}
\def\disp#1{{\displaystyle #1}}
\def\J{{\mathcal J}}
\def\P{{\mathcal P}}
\def\C{{\mathcal C}}
\def\E{{\mathcal E}}
\def\H{{\mathcal H}}

\def\odelta{{\overline{\Delta}}}

\def\Int{{\rm Int}}
\def\ph{\varphi}
\newcommand {\RR} {\mathbb R}

\newcommand {\defeq}    {\stackrel{\rm def}{=}}
\newcommand {\PROBA}[1] {\mathbb P\left(#1\right)}
\newcommand {\EXPECT}[1] {\mathbb E\left( #1 \right)}
\newcommand {\LEBEG}[2] {{\mathbb L}^{^{_{\!\!#2}}}\big(#1\big)}
\newcommand {\COMB}[2] {\frac{#2 !}{#1 ! (#2 - #1)!}}

\newcommand {\ASTK}[2] {{ {\mathcal I}^{#2}\left( #1 \right) }}
\newcommand {\ASTKE}[2] {{ {\mathcal D}^{#2}\left( #1 \right) }}
\newcommand {\TTJ}[1] {{ {\mathcal T}_{ #1 }}}

\newcommand {\TOJ}[1] {{ {\mathcal T}_{ #1}^\bot }}
\newcommand {\TTT}[2] { {\mathcal T}^{#2}_{ #1} }
\newcommand {\CK}[1]  {\mbox{\boldmath$\overline{C}$\unboldmath$_{#1}$}}
\newcommand {\QK}[1]  {\mbox{\boldmath$\overline{Q}$\unboldmath$_{K,#1}$}}

\newcommand {\PS}[2] {\langle #1 , #2 \rangle}

\newcommand {\LS}[2] {{\mathcal L}_{#1}(#2)}

\newcommand {\JACOB}[1] {{ [ \! [   #1  ] \! ] }}

\DeclareMathOperator{\val}{val}
\DeclareMathOperator{\Span}{span}

\newtheorem{theorem}{Theorem}
\newtheorem{proposition}{Proposition}

\newtheorem{lemma}{Lemma}
\newtheorem{corollary}{Corollary}
\newtheorem{remark}{Remark}

\def\DESSINUN{
\psfrag{tt1}{$\partial\TTT{\{1\}}{\tau}=\partial\TTT{\{4\}}{\tau}$}
\psfrag{tt2}{$\partial\TTT{\{2\}}{\tau}$}
\psfrag{tt3}{$\partial\TTT{\{3\}}{\tau}$}
\psfrag{p1}{$\psi_1$}
\psfrag{p2}{$\psi_2$}
\psfrag{p3}{$\psi_3$}
\psfrag{p4}{$\psi_4$}
\psfrag{po2}{$P_{\TOJ{\{2\}}}(\LS{\|.\|}{\tau})$}
\psfrag{po4}{$P_{\TOJ{\{4\}}}(\LS{\|.\|}{\tau})$}
\psfrag{po3}{$P_{\TOJ{\{3\}}}(\LS{\|.\|}{\tau})$}
\psfrag{ls}{$\LS{\|.\|}{\tau}=\TTT{\emptyset}{\tau}
  =\ASTK{0}{\tau}$}
\psfrag{lsd}{$\partial \LS{f_d}{\theta}$}
\begin{figure}
\centerline{\hbox{\includegraphics[width=14.5cm]{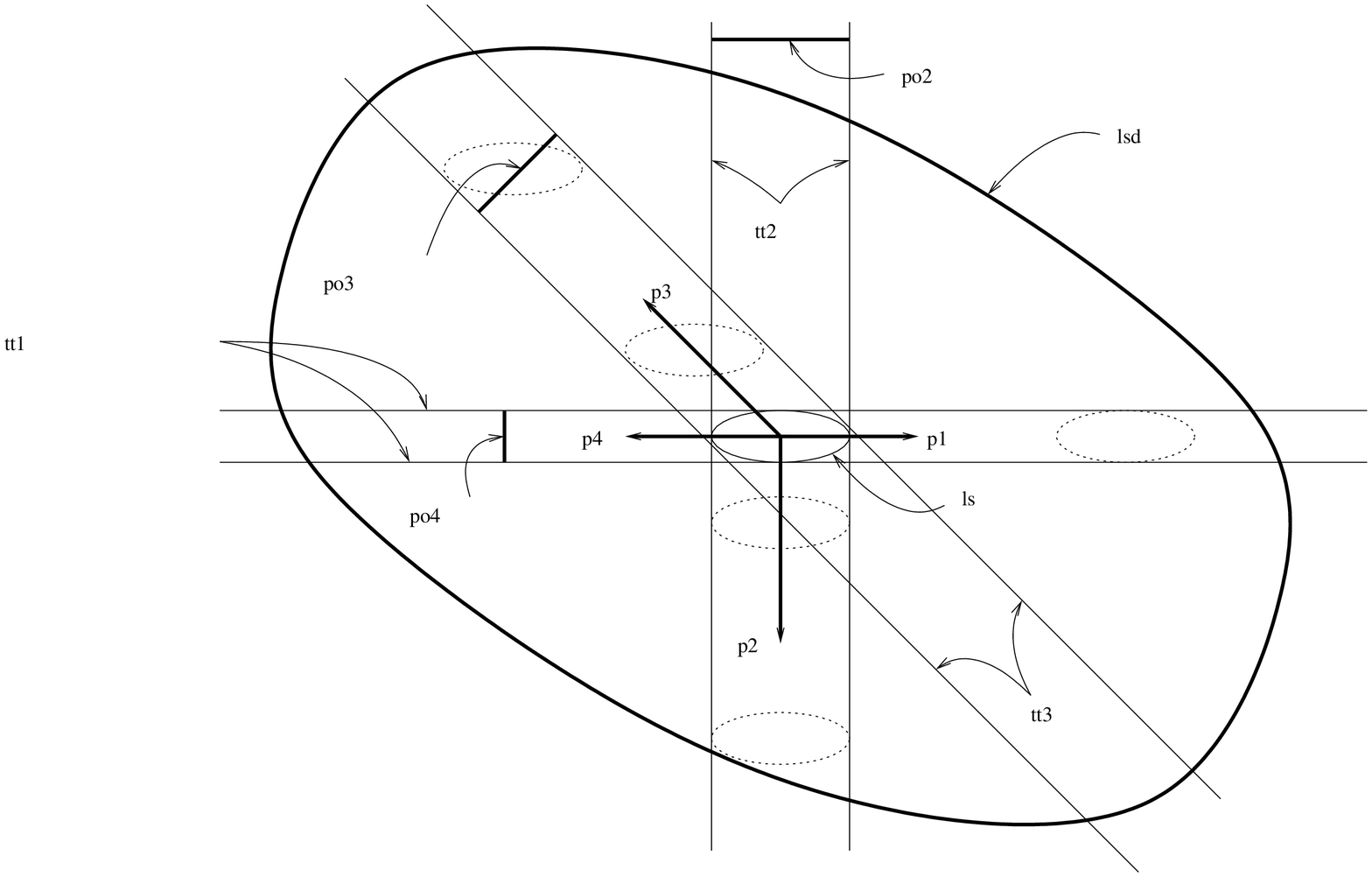}}}
\caption{\label{dessin1}
Example in dimension $2$. Let the dictionary read~
$\{\psi_1,\psi_2,\psi_3,\psi_4\}$. On the drawing, the sets
$P_{\TOJ{\{i\}}}(\LS{\|.\|}{\tau})$, for $i=2, 3, 4$, are shifted by an
element of $\TTJ{\{i\}}$. The dotted sets represent
translations of $\LS{\|.\|}{\tau}$. The set-valued function $\ASTK{}{\tau}$,
as presented in (\ref{ASTK}) and Proposition~\ref{astk=sstk}, gives rise to the following situations:
$\ASTK{0}{\tau}=\LS{\|.\|}{\tau}=\TTT{\emptyset}{\tau}$,
$\ASTK{1}{\tau} = \TTT{\{1\}}{\tau}\cup \TTT{\{2\}}{\tau}\cup
\TTT{\{3\}}{\tau}$ and $\ASTK{2}{\tau} =\RR^2=\TTT{\{1,2\}}{\tau}=\TTT{\{2,3\}}{\tau}=\ldots$
The symbol  $\partial$ is used to  denote the boundaries of the sets.
}
\end{figure}
}

\def\INTERSECTION{
\psfrag{V12}{$\TTJ{\{1,2\}}$}
\psfrag{V34}{$\TTJ{\{3,4\}}$}
\psfrag{V1234T}{$\TTT{\{1,2\}}{\tau} \cap \TTT{\{3,4\}}{\tau}$}
\psfrag{V1234}{$W = \TTJ{\{1,2\}} \cap \TTJ{\{3,4\}}$}
\psfrag{p1}{$\psi_1$}
\psfrag{p2}{$\psi_2$}
\psfrag{p3}{$\psi_3$}
\psfrag{p4}{$\psi_4$}
\begin{figure}
\centerline{\hbox{\includegraphics[width=17cm]{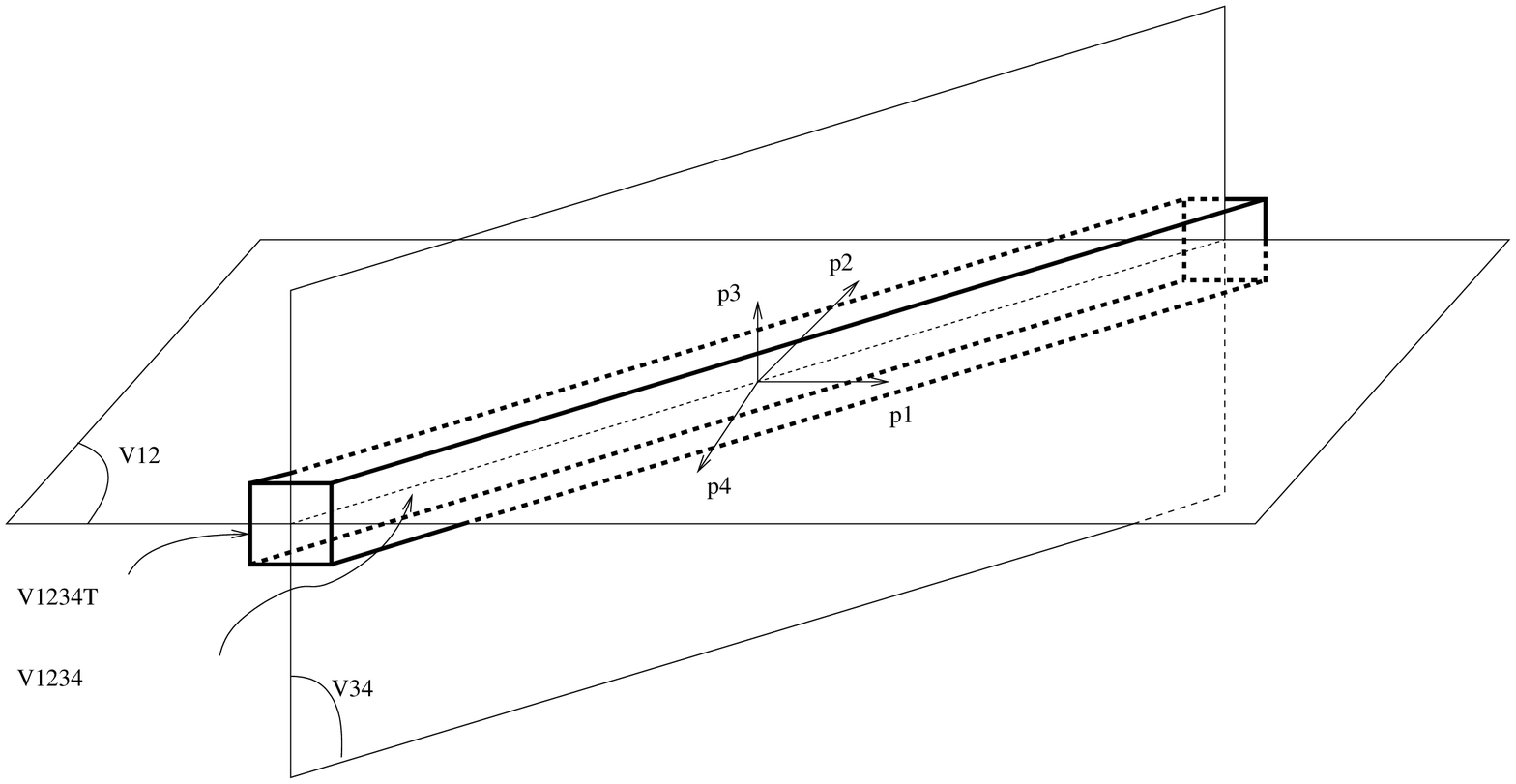}}}
\caption{\label{intersection}
Example of an intersection in dimension $3$. $\TTT{\{1,2\}}{\tau}$ is in between to planes, parallel to
$\TTJ{\{1,2\}}$. Same remark for $\TTT{\{3,4\}}{\tau} $. The set
$\TTT{\{1,2\}}{\tau} \cap \TTT{\{3,4\}}{\tau}$ is of
the form $ W + P_{W^\bot}\LS{\tilde g}{\tau}$, where $\tilde g$ is a
norm and for $W= \TTJ{\{1,2\}} \cap \TTJ{\{3,4\}}$. We also have
$\dim(\TTJ{\{1,2\}} \cap \TTJ{\{3,4\}}) < dim(\TTJ{\{1,2\}})=dim(\TTJ{\{3,4\}})$.
}
\end{figure}
}

\title{Average performance of the sparsest approximation using a general dictionary}
\author{Fran\c{c}ois Malgouyres$^\star$ ~and~ Mila Nikolova$^\diamond$}
\date{$^\star$  LAGA/L2TI, Universit\'e Paris 13, CNRS, 99 avenue
  J.B. Cl\'ement, 93430 Villetaneuse, France;\\
 (33/0) 1-49-40-35-83, malgouy@math.univ-paris13.fr\\
$^\diamond$ CMLA, ENS Cachan, CNRS, PRES UniverSud, 61 Av. President Wilson,
  F-94230 Cachan, France\\
(33/0) 1 47 50 59 08
nikolova@cmla.ens-cachan.fr\\ ~ \\ \today }
\begin{document}
\maketitle

\addtolength{\baselineskip}{2mm}


{\bf AMS class:} 41A25, 41A29, 41A45, 41A50, 41A63.

\begin{abstract}
We consider the minimization of the number of non-zero coefficients
(the $\ell_0$ ``norm'') of the representation of a data set in terms
of a dictionary under a fidelity constraint. (Both the dictionary and
the norm defining the constraint are arbitrary.) This (nonconvex)
optimization problem naturally leads to the sparsest representations,
compared with other functionals instead of the  $\ell_0$ ``norm''.

Our goal is to measure the sets  of data yielding a $K$-sparse solution---i.e.
involving $K$ non-zero components. Data are assumed uniformly
distributed on a  domain defined by any norm---to be chosen by the
user. A precise description of these sets of data is given and
relevant bounds on the Lebesgue measure of these sets are derived.
They naturally lead to bound the probability of getting a $K$-sparse
solution. We also express the expectation of the number of non-zero components.
We further specify these results in the case of the Euclidean norm,
the dictionary being arbitrary.
% Je n'aime pas trop ce genre de phrase (mets la si tu y tiens)
%These results have a considerable
%theoretical and practical importance.

% Je l'ai mis dans les perspectives (cela me semble etre le bon endroit)
%In a forthcoming work, these results are speciaized to the context of
%orthogonal bases.
\end{abstract}

\section{Introduction}
\subsection{The problem under consideration}

Our goal is to represent observed data $d\in\RR^N$ in a economical way using a dictionary $(\psi_i)_{i\in I}$ on $\RR^N$,
where $I$ is a finite set of indexes and
\beq \Span\big\{ \psi_i:i\in I \big\}=\RR^N  .\label{spanRN}\eeq
We study the sparsest representation where the
(unknown) coefficients  $(\lambda_i)_{i\in I}$ are estimated
%The unknown is a set of  coefficients $(\lambda_i)_{i\in I}$ where $I$ is a finite set of indexes.
%We study their estimation based on data $d\in\RR^N$
 by solving the constraint optimization problem
$(\P_d)$  given below:
\beq(\P_d ): ~~~~~~~~~~~~~~~~~~~~\left\{\begin{array}{l}
\disp{\mbox{\rm minimize}_{(\lambda_i)_{i\in I}}} \ell_0\big((\lambda_i)_{i\in I}\big), \\
\mbox{under the constraint~: ~~} \left\|\disp{\sum_{i\in I}} \lambda_i \psi_i -  d\right\|\leq
\tau,
\end{array}\right.~~~~~~~~~~~~~~~~~~~~~~~~~~~~~~~~
\label{Pd}\eeq
with
\[\ell_0((\lambda_i)_{i\in I}) \defeq \#\big\{i\in I: \lambda_i\neq 0\big\},\]
where $\#$ stands for cardinality, $\|.\|$ is an arbitrary norm and  $\tau>0$ is a fixed parameter.
Let us emphasize  that for any $d\in\RR^N$, the constraint in $(\P_d)$ is nonempty thanks to \eqref{spanRN}
and that the minimum is reached since $\ell_0$ takes its values in the finite set $\{0,1,\ldots,\#I\}$.

Given the data $d$, the norm $\|.\|$, the parameter $\tau$ and the
dictionary, the solution of $(\P_d)$ is the sparsest possible, since the objective function $\ell_0$ in (\ref{Pd}) minimizes the number of all
non-zero coefficients in the set  $(\lambda_i)_{i\in I}$ without penalizing them.

The function $\ell_0$ is sometimes abusively called the $\ell_0$-norm.
It can equivalently be written as
\beq \sum_{i\in I}\ph(\lambda_i)~~~\mbox{where}~~~
\ph(t)=\left\{\barr{ccc} 0 &\mbox{if}& t=0\\
1&\mbox{if}& t\neq 0
\earr\right.\label{ph}\eeq
The function $\ph$ is  discontinuous  at zero and $\C^\infty$ beyond the origin, and has a long history.
It was used in the context of Markov random fields by Geman and Geman 1984, cf. \cite{Geman84} and Besag 1986 \cite{Besag86}
as a prior in MAP energies to restore labeled images (i.e. each $\lambda_i$ belonging to a finite set of values):
\beq\E(\lambda)=\big\|\sum_{i\in I}\lambda_i \psi_i -  d\big\|_2^2+\beta\sum_{i\sim j}\ph(\lambda_i-\lambda_j),\label{MAP}\eeq
where the last term in (\ref{MAP})
counts the number of all pairs of dissimilar neighbors $i$ and $j$, and $\beta>0$ is a parameter.
This label-designed form is known as the Potts prior model, or as the multi-level logistic model \cite{Besag89,Li95}.
Guided by the {\it Minimum description length} principle of Rissanen, Y. Leclerc proposed in 1989 in \cite{Leclerc89}
the same prior to restore piecewise constant, real-valued images.
The hard-thresholding method to restore noisy wavelet coefficients, proposed by Donoho and Johnstone in 1992, see \cite{Donoho92},
amounts to minimize for each coefficient $\lambda_i$ a function of the form $\|\lambda_i-g_i\|_2^2+\beta\ph(\lambda_i)$ where
the noisy coefficients read $g_i=\langle\psi_i^*,d\rangle$, $\forall i\in I$ where $(\psi_i)_{i\in I}$ is a wavelet basis.
Very recently, the energy \eqref{MAP} was successfully used to reconstruct 3D tomographic images by using stochastic continuation by
Robini and Magnin \cite{Robini07}.
Let us notice that even though the problem $(\P_d)$ in (\ref{Pd}) and the minimization of $\E$ in (\ref{MAP}) are closely related,
there is no rigorous equivalence in general.

The context of digital image compression is of a particular interest, since it
is typically the problem we are modeling in the paper. In compression,
one considers different classes of images. Those digital images  live in $\RR^N$ and are
obtained by sampling an analogue image. Their distribution
in $\RR^N$ is one of the main unknown in image processing and, in
practice, we only know some realizations of this distribution
(i.e. some images). Given this (unknown) distribution, the goal of
image compression is to build a coder (that encodes elements of
$\RR^N$) which assigns a small code to images. Typically, we want for
every image $d\in\RR^N$
\[\PROBA{length(code(d)) = K}
\]
to be as large as possible for $K$ small, and small for $K$ large. We
also want the decoder to satisfy $decode(code(d)) \sim d$.

The link with the problem $(\P_d)$, in \eqref{Pd}, is that the
current image compression standards (JPEG, JPEG2000) encode quantized versions
of the coordinates of the image in a given basis. Moreover, most of
the gain is made by choosing a basis such that the number of non-zero
coordinates (after the quantization process) is small (\cite{GormishLeeMarcellinJPEG2000,WallaceJPEG}). That is, we
want to solve $(\P_d)$ for each $\lambda_i$ belonging to a finite set of
values and for a basis $(\psi_i)_{i\in I}$. This link between image
compression and $(\P_d)$ might seem restrictive when we only consider a
basis. It makes much more sense when we consider a redundant system of
vectors $(\psi_i)_{i\in I}$. The use of redundant dictionaries has
known a strong development in the past years, see \cite{CoifmanWickerhauser,MallatZhang,pati93OMP,ChenDonoho} for the
most  famous examples. In the context of dictionaries, we know that the length of
the code for encoding $(\lambda_i)_{i\in I}$ is in general
proportional to $\ell_0((\lambda_i)_{i\in I})$. The problem $(\P_d)$
therefore reads : minimize the codelength of the image while
constraining a given level of accuracy of the coder. This is exactly
the goal in image compression.

Finding an exact solution to $(\P_d)$ in large dimension
(which is necessary in order to apply $(\P_d)$ to image compression)
still remains a challenge.
In fact, the methods described in
\cite{CoifmanWickerhauser,pati93OMP,ChenDonoho} can be
seen as heuristics approximating $(\P_d)$. The links between the
performances of those heuristics and the performances of $(\P_d)$ is
not completely clear. It is also a goal of the paper to provide a mean
for comparing those algorithms.

%For the greedy algorithms (see \cite{MallatZhang,pati93OMP}), we refer to
%\cite{BarronCohenDahmenDeVore,Temlyakov}. Typically, where the norm of
%the error ($\|d-\sum_{i\in I} \lambda_i \psi_i\|_2$) is bounded from above by a function of the number of
%non-zero coordinates in $(\lambda_i)_{i\in I}$.
%For the Basis Pursuit Denoising algorithm,
% {\bf
% Dans \cite{Donoho06a} je lis:
% There is by now an extensive literature exhibiting results on equivalence of `1 and `0 minimization
% \cite{Donoho01g,Elad02,Tropp04,Tropp06,Fuchs04}. Ces refs sont dans la biblio, je te laisse trier.
% }

\subsection{Our contribution}

In this paper, we estimate the ability of the model $(\P_d)$ to
provide a sparse representation of data which follows a given
distribution law. The distribution law is uniform in the
$\theta$-level set of a norm $f_d$ :
\[\LS{f_d}{\theta}=\{w\in\RR^N, f_d(w)\leq \theta\}.\]

In order to do this we
\begin{itemize}
\item Give a precise (and non redundant) geometrical description of the sets
\[\ASTK{K}{\tau} =\left\{d\in\RR^N, \val(\P_d)\leq K\right\},
\]
and
\begin{equation}\label{poatnb}
\ASTKE{K}{\tau} = \left\{d\in\RR^N, \val(\P_d) = K \right\}
\end{equation}
where $\val(\P_d)$ denotes
$\ell_0((\lambda_i)_{i\in I})$ for a solution $(\lambda_i)_{i\in I}$ of
$(\P_d)$ and for $K=0,\ldots,N$, $\tau>0$. This is done in Theorem \ref{proj=all} and equation \eqref{diff1}.

\begin{remark} It is easy to see that
$\Big\{\psi_i: \lambda_i\neq 0~\mbox{for}~( \lambda_i)_{i\in I}~\mbox{\rm solving}~(\P_d)\Big\}$
forms a set of linearly independent vectors. Therefore for all $d\in\RR$ we will find a solution
with at most $N$ nonzero coefficients, even if the size of the dictionary is huge, $\#I \gg  N$.
So in this work we consider solutions with sparsity $K\leq N$.
\label{REM1}\end{remark}
\item Once these sets are precisely described, we are able to
  bound (both from above and from below), their measure (more precisely the
  measure of their intersection with $\LS{f_d}{\theta}$). The
  difference between the upper and the lower bound is negligible when
  compared to $\left(\frac{\tau}{\theta}\right)^{N-K}$, when
  $\frac{\tau}{\theta}$ is mall enough. Moreover,
  these bounds show that the measures of
  $\ASTK{K}{\tau}\cap\LS{f_d}{\theta}$ and
  $\ASTKE{K}{\tau}\cap\LS{f_d}{\theta}$ asymptotically behave like
  \[\CK{K} \theta^N\left(\frac{\tau}{\theta}\right)^{N-K},\]
  as $\frac{\tau}{\theta}$ goes to $0$.

  The constants $\CK{K}$ are defined in \eqref{CK}. They are made of
  the sum of constants $C_V$ over all  possible vector
  subspaces $V$ of dimension $K$, spanned by elements of the dictionary
  $(\psi_i)_{i\in I}$. The constants $C_V$ are built in Proposition
  \ref{propV} and Corollary \ref{rmkTJ}. They have the form
  \[C_V =\LEBEG{P_{V^\bot}\left(\LS{\|.\|}{1}\right)}{N-K}
  \LEBEG{V\cap\LS{f_d}{1}}{K},
  \]
  where $P_{V^\bot}$ is the orthogonal projection onto the
  orthogonal complement of $V$, $\|.\|$ is the norm defining the data
  fidelity term in $(\P_d)$ and $\LEBEG{.}{k}$ denotes the Lebesgue
  measure of a set living in $\RR^k$.
\item Once this is achieved, we easily obtain lower and upper bounds
  for $\PROBA{\val(\P_d)\leq K}$, $\PROBA{\val(\P_d) = K}$ when $d$ is
  uniformly distributed in $\LS{f_d}{\theta}$ (see Section
  \ref{stat}). They have the same characteristics as the bounds
  described above (modulo the disappearance of $\theta^N$). In order to obtain sparse representations of the
  data, we should therefore tune the model (the norm $\|.\|$ and the
  dictionary $(\psi_i)_{i\in I}$) in order to obtain larger constants $\CK{K}$.

  This result clearly shows that the model $(\P_d)$ benefits from several
  ingredient (which might not be present in other models promoting sparsity):
  \begin{itemize}
  \item the sum defining $\CK{K}$ is for \underline{all} the possible
    vector subspaces of dimension $K$ spanned by elements of the
    dictionary $(\psi_i)_{i\in I}$.
  \item the term  $\LEBEG{V\cap\LS{f_d}{1}}{K}$ in in the constants $C_V$ represents
    the measure of the \underline{whole} set
    $V\cap\LS{f_d}{1}$.
  \end{itemize}
\item Finally we estimate $\EXPECT{\val(\P_d)}$ and show that its
  asymptotic (when $\frac{\tau}{\theta}$ goes to $0$) is
  governed by the constant $\CK{N-1}$ (see Theorem
  \ref{thmexpect}). Increasing this constant therefore seems to be
  particularly important when building a model $(\P_d)$ (i.e. choosing
  $\|.\|$ and $(\psi_i)_{i\in I}$).
\end{itemize}

These results are illustrated in the context of particular choice for $\|.\|$ and for
$f_d$ in Section \ref{illustre-sec}.

\subsection{Relation to other evaluations of performance}

Evaluating the performance of an optimization problem like  $(\P_d)$ for the purpose
of realizing \underline{nonlinear approximation} is a very active firld of research.
%The evaluation of the performance of $(\P_d)$ is a very developed and
%establish field of research. It is generally named ``Nonlinear
%approximation''. It is difficult to properly describe this field in few
For a good survey of the problem we refer to \cite{Devore98}.

In that field of research a variant of $(\P_d)$, named ``best K-term
approximation'', is under study. It consists in looking for the best
possible approximation of a datum $d\in\RR^N$ using an expansion in
$(\psi_i)_{i\in I}$ with $K$ non-zero coordinates. The performance of
the model is estimated using the quantity
\[\sigma_K(d) = \inf_{S\in \Sigma_K} \|d - S\|,
\]
where $\Sigma_K$ denotes the union of all the vector spaces of
dimension $K$ spanned by elements of $(\psi_i)_{i\in I}$, for
$K=0,\ldots, N$. Expressed with our notations, the typical object under
consideration is\footnote{In Nonlinear approximation authors usually
  consider infinite dimensional spaces.}
\[{\mathcal A}^\alpha(C) = \bigcup_{K=1}^N \ASTKE{K}{\frac{C}{K^\alpha}},
\]
for $C>0$ and $\alpha>0$ and $\ASTKE{K}{\tau}$ defined by \eqref{poatnb}. That is the data $d$ obeying
\[\sigma_K(d) \leq \frac{C}{K^\alpha}~~~~\mbox{, for all
}K=1,\ldots,N.
\]

The typical results obtained there take the form
\begin{equation}\label{aerbon}
{\mathcal A}^\alpha(C_1) \subset {\mathcal K}_{\eta}\subset{\mathcal
  A}^\alpha(C_2),
\end{equation}
for $C_2 \geq C_1>0$ and the level set
\[{\mathcal K}_{\eta} = \{d\in \RR^N, \|d\|_\eta \leq 1\},
\]
for a norm $\|.\|_\eta$ characterizing the regularity of $d$ (again,
the theory is in infinite dimensional vector spaces). This permits to
estimate the number of coordinates which are needed to represent a datum
$d$, if we know its regularity. Typically, the link between $\alpha$
and $\eta$ says how good is the basis (or more generally a dictionary) at representing
the data class.

The clear advantage of these results over ours is that they apply even
if one only has a vague knowledge of the data distribution. For
instance, any data distribution whose support is included in
${\mathcal K}_{\eta}$ does enjoy the decay $\frac{C_2}{K^\alpha}$. The
inclusions in \eqref{aerbon} need indeed to be true for the worse
elements of ${\mathcal K}_{\eta}$ (even if they are rare). The
counterpart of this advantage is that the constants $C_1$ and $C_2$
might be pessimistic.

% Another advantage of Nonlinear approximation is that it leads to very precise relations between
% the regularity of the data and the number of coordinates required in its
% approximation. In particular because the researchers of this field study data living in
% infinite dimensional spaces, if the relation is not met nothing can
% be said in general. The counterpart of this advantage is that this kind of
% analysis is only valid for a limited number of scenarios (when the
% approximation works well).

Finally, as far as we know, the analysis proposed in Nonlinear
approximation does not permit (today) to clearly assess the
differences between $(\P_d)$ and its heuristics (in particular Basis
Pursuit Denoising \cite{ChenDonoho} and Orthogonal Matching Pursuit
\cite{pati93OMP}). This is a clear advantage of the method for
assessing model performances proposed in this paper. Indeed,
similar analysis have already been conducted in
\cite{Malgouyres06a,Malgouyres-compres-estim-proof,MalgouyresSigPro}
in the context of the compression scheme described in
\cite{MalgouyresCompression}, Basis Pursuit Denoising and total
variation regularization. (However, concerning the papers on Basis Pursuit Denoising and the
total variation regularization, the results are stated for another
asymptotic and the analysis partly needs to be rewritten in the
proper context.)

\subsection{Notations}

For any function $f:\RR^N\to\RR$, and any
$\theta\in\RR$, the $\theta$-level set of $f$ is denoted by
\begin{equation}\label{ls}
\LS{f}{\theta} = \{w\in\RR^N, f(w)\leq \theta\}.
\end{equation}
For any vector subspace $V$ of $\RR^N$, we denote $P_V$ the
orthogonal projection onto $V$ and by $V^\bot$ the orthogonal complement
of $V$ in $\RR^N$.
To specify the dimension of $V$, we write $\dim(V)$.
The Euclidean norm of an $u\in\RR^N$ is systematically denoted by
$\|u\|_2$. The notation $\|u\|$ is devoted to a general norm on $\RR^N$.
For any integer $K>0$, the Lebesgue measure on $\RR^K$ is
systematically denoted by~$\LEBEG{.}{K}$, whereas $I_K$ stands  for the $K\times K$ identity matrix.
We write $\PROBA{.}$ for probability and $\EXPECT{.}$ for expectation.

As usually, we write $o(t)$ for a function satisfying $\lim_{t\to 0}\frac{o(t)}{t}=0$.

For any $d\in\RR^N$, we denote $\val(\P_d)$ the value of the minimum in
$(\P_d)$---i.e. $\ell_0\left((\lambda_i)_{i\in I}\right)$ for
$(\lambda_i)_{i\in I}$ solving $(\P_d)$.

\section{Measuring bounded cylinder-like subsets of $\RR^N$}\label{mezcyl}

\subsection{Preliminary results}

Below we give  several statements that will be used many times in the rest of the work.

\begin{lemma}\label{lemme_h}
For any vector subspace $V\subset\RR^N$ and any norm $\|.\|$ on
$\RR^N$, define the application
\beqn h:V^\bot&\to&\RR\nonumber\\
u&\to&h(u)\defeq\inf\left\{t\geq 0:\frac{u}{t}\in P_{V^\bot}\left(\LS{\|.\|}{1}\right)\right\}.
\label{hu}
\eeqn
Then the following holds:
\bit
\item[(i)] For any $\tau\geq 0$,  we have
\beq\LS{h}{\tau}=P_{V^\bot}\left(\LS{\|.\|}{\tau}\right).
\label{h}\eeq

\item[(ii)] The application $h$  in \eqref{hu} is a norm on $V^\bot$.
\item[(iii)]   For any norm $f_d$ on $\RR^N$, let $\delta_1>0$, $\delta_2>0$ and $\odelta$ be some constants satisfying
\beqn w\in \RR^N&\Rightarrow&f_d(w)\leq \delta_1 \|w\|_2 \mbox{~ and ~}  \|w\|_2\leq \delta_2
  \|w\|,
\label{Y}\\
\odelta &\defeq& \delta_1\delta_2 \label{odelta}
\eeqn
The constants $\delta_1$, $\delta_2$ and $\odelta>0$ are independent of $V$ and we have
\beqn \label{major1}
    f_d(u) &\leq& \odelta h(u),~~~~~~\forall u\in V^\bot,
  \\
\label{major2}
    \|u\|_2 &\leq &\delta_2 h(u),~~~~~~\forall u\in V^\bot.
  \eeqn
\eit
\end{lemma}

\begin{remark}
The constants in (\ref{Y}) come from the fact that all norms on  a finite-dimensional space are equivalent.
In practice we will choose the smallest constants satisfying these inequalities.
\label{findim}\end{remark}

\begin{proof}
The case $V=\{0\}$ is trivial (we obtain $h=\|.\|$) and we further assume that $\dim(V)\geq 1$.

\paragraph{\it Assertion (i).} The set $P_{V^\bot}(\LS{\|.\|}{1})$ is
convex since $\|.\|$ is a norm and $P_{V^\bot}$ is linear.
Moreover, the origin 0 belongs to its interior.
Indeed, there is $\eps>0$ such that if $w\in\RR^N$ satisfies
$\|w\|_2<\eps$, then $\|w\|<1$.
Consequently  $0\in \Int\big(\LS{\|.\|_2}{\eps}\big)\subset\LS{\|.\|}{1}$.
Using that $\|.\|_2$ is rotationally invariant and that $P_{V^\bot}$ is a contraction, we deduce
that $0\in \Int\big(P_{V^\bot}(\LS{\|.\|_2}{\eps})\big)\subset P_{V^\bot}(\LS{\|.\|}{1})$.
Then the application $h:V^\bot\to\RR$ in \eqref{hu} is the usual Minkowski functional
 of $P_{V^\bot}\left(\LS{\|.\|}{1}\right)$, as defined and commented in \cite[p.131]{Luenberger69}.
Since $P_{V^\bot}\left(\LS{\|.\|}{1}\right)$ is closed, we have
\[P_{V^\bot}\left(\LS{\|.\|}{1}\right)=\big\{ u\in V^\bot : h(u)\leq 1\big\}.\]
Using that the Minkowski functional is positively homogeneous---i.e.
$$h(\tau u)=\tau h(u),~~~\forall\tau>0,$$
lead to~(\ref{h}).

\paragraph{\it Assertion (ii).} For  $h$ to be a norm, we have to show that the latter property holds for
any $\lambda\in \RR$ (i.e.  that $h$ is symmetric with respect to the
origin). It is true since, for any $\lambda\in \RR$
\beqnn ~~~~~~~~~ h(\lambda u)&=&\inf\big\{t\geq0 : \lambda u\in P_{V^\bot}\left(\LS{\|.\|}{t}\right)\big\}\\
&=&\inf\big\{t\geq0 : u\in P_{V^\bot}\left(\LS{\|.\|}{\frac{t}{|\lambda|}}\right)\big\}\\
&=&|\lambda|\inf\big\{t\geq0 : u\in P_{V^\bot}\left(\LS{\|.\|}{t}\right)\big\}~~~~~~~~~\mbox{(writing $t$ for $t/|\lambda|$)}\\
&=&|\lambda|~ h(u),
\eeqnn
where we use the facts that $P_{V^\bot}$ is linear and that $\|.\|$ is
a norm.
It is well known that  the Minkowski functional is non negative, finite, and satisfies\footnotemark~
\footnotetext{For completeness, we give the details:
\beqnn h(u+v)&=&\inf\big\{t\geq0:(u+v)\in P_{V^\bot}\left(\LS{\|.\|}{t}\right)\big\}\\
&\leq&\inf\big\{t\geq0:u\in P_{V^\bot}\left(\LS{\|.\|}{t}\right)\big\}+
\inf\big\{t\geq0:v\in P_{V^\bot}\left(\LS{\|.\|}{t}\right)\big\}=h(u)+h(v).
\eeqnn
}
$h(u+v)\leq h(u)+h(v)$ for any $u,v \in V^\bot$.

Finally, since
  $\LS{h}{0}=P_{V^\bot}\left(\LS{\|.\|}{0}\right)=\{0\}$,
\[h(u)=0~~~\Leftrightarrow~~~u=0.
\]
Consequently, $h$ defines a norm on $V^\bot$.

\paragraph{\it Assertion (iii).} Let us first remark that
\[ \LS{\|.\|}{1} \subset \LS{\|.\|_2}{\delta_2} \subset
\LS{f_d}{\delta_1 \delta_2}=\LS{f_d}{\odelta},
\]
where $\delta_1$ and $\delta_2$ are defined in the proposition.
Using that $\|.\|_2$ is rotationally invariant, we have
\beqn
\LS{h}{1}=P_{V^\bot}\left(\LS{\|.\|}{1}\right) & \subset &
P_{V^\bot}\left(\LS{\|.\|_2}{\delta_2}\right) = \LS{\|.\|_2}{\delta_2} \cap V^\bot
\nonumber\\
& \subset & \LS{f_d}{\delta_1 \delta_2}  \cap V^\bot =\LS{f_d}{\odelta}  \cap V^\bot.
\nonumber\eeqn

We will prove \eqref{major1} and  \eqref{major2} jointly. To this end
let us consider a norm $g$ on $\RR^N$ and $\delta>0$ such that
\begin{equation}\label{Rgener}
P_{V^\bot}\left(\LS{\|.\|}{1}\right) \subset \LS{g}{\delta} \cap V^\bot.
\end{equation}
Using that each norm can be expressed as a Minkowski functional, for any $u\in V^\bot$ we can write down the following:
\beqn
g(u) & = & \inf\{t\geq 0 : g(\frac{u}{t}) \leq 1\} \nonumber\\
  & = & \inf\{t\geq 0 : g(\frac{\delta }{t} u) \leq \delta\} \nonumber\\
   & = & \delta \inf\{t\geq 0 : g(\frac{u}{t}) \leq \delta \} \nonumber
   ~~~~~~~~~~~~~\mbox{(write $t$ for $\frac{t}{\delta}$)}\\
  & = & \delta \inf\{t\geq 0 : \frac{u}{t} \in\LS{g}{\delta}\} \nonumber\\
  & \leq & \delta \inf\{t\geq 0 : \frac{u}{t}
  \in P_{V^\bot}\left(\LS{\|.\|}{1}\}\right) \label{Q}\\
  & \leq & \delta~ h(u),\nonumber
\eeqn
where the inequality in (\ref{Q}) comes from (\ref{Rgener}).

If we identify $g$ with $f_d$ and $\delta$ with $\odelta$, we obtain \eqref{major1}.
Similarly, identifying $g$ with $\|.\|_2$ and $\delta$ with $\delta_2$ yields  \eqref{major2}.
This concludes the proof.
\end{proof}

The next proposition addresses sets of $\RR^N$ bounded with the aid of $f_d$.
\begin{proposition}\label{propV}
For any vector subspace $V$ of $\RR^N$, any norm $\|.\|$ on $\RR^N$ and any $\tau>0$, define
\beq V^\tau =  V+ P_{V^\bot}\left(\LS{\|.\|}{\tau}\right).
\label{Vtau}\eeq
Then the following hold:
\bit
\item[(i)] $V^\tau$ is closed and measurable;

\item[(ii)] Let $f_d$ be any norm on $\RR^N$, $h:V^\bot\to\RR$ the norm defined in Lemma \ref{lemme_h},
$K=\dim(V)$
and $\delta_V$ be any constant  such that
%for any norm $f_d$ on $\RR^N$, any constant $\delta_V$ such that
\beq f_d(u)\leq\delta_{_V} h(u),~~~\forall u\in V^\bot.\label{dV}\eeq
%for $h$ defined in Lemma \ref{lemme_h}  and
If ~$\theta\geq\delta_{_V}\tau$, then
\beq C\tau^{N-K}(\theta-\delta_{_V}\tau)^K\leq
\LEBEG { V^\tau \cap \LS{f_d}{\theta}}{N}\leq C\tau^{N-K}(\theta+\delta_{_V}\tau)^K,\label{encad}\eeq
where
\beq C=\LEBEG{P_{V^\bot}\left(\LS{\|.\|}{1}\right)}{N-K} ~\LEBEG{V\cap\LS{f_d}{1}}{K}
%C=\int_{P_{V^\bot}\left(\LS{\|.\|}{1}\right)}du\int_{V\cap\LS{f_d}{1}}dv
~\in~ (0,+\infty).\label{conC}
\eeq
\eit
\end{proposition}

\begin{remark}
Using Lemma  \ref{lemme_h}, the condition in (\ref{dV}) holds  for any $\delta_V\geq\delta_V^*$ with
$\delta_V^*\in [0,\odelta]$, where $\odelta$ is given in \eqref{odelta}.
Let us emphasize that $\delta_V$ may depend on $V$ (which explains the letter ``V'' in index).
The proposition clearly holds if we take  $\delta_V= \odelta$---the constant of Lemma \ref{lemme_h}, assertion (iii), which is
independent of the choice of~$V$.

Observe that $C$ is a positive, finite constant that depends only on $V$, $\|.\|$ and $f_d$.
\end{remark}
\begin{remark}
An important consequence of this proposition is that asymptotically
\[\LEBEG{V^\tau \cap \LS{f_d}{\theta}}{N} = C \theta^N \left(\frac{\tau}{\theta}\right)^{N-K}
+ \theta^N
o\left(\left(\frac{\tau}{\theta}\right)^{N-K}\right)~~~\mbox{if}~~~\frac{\tau}{\theta}\to
0.
\]
\end{remark}

\begin{proof}
The sets $V$ and $P_{V^\bot}\left(\LS{\|.\|}{\tau}\right)$ are
closed. Moreover, $V$ and $P_{V^\bot}\left(\LS{\|.\|}{\tau}\right)$
are orthogonal. Therefore $V^\tau$ is closed. As a consequence
$V^\tau$ is a Borel set and is Lebesgue measurable.

Since the restriction of $f_d$ to $V^\bot$ is a norm on $V^\bot$,
there exists $\delta_V$ such that (see Remark \ref{findim})
\beq f_d(u)\leq\delta_{_V} h(u),~~~\forall u\in
V^\bot,\label{delV}\eeq
where $h$ is given in (\ref{h}) in Lemma \ref{lemme_h}.
By (\ref{major1}) in Lemma \ref{lemme_h}, such a $\delta_V$ exists in $[0,\odelta]$.
To simplify the notations, in the rest of the proof we will write $\delta$ for $\delta_{_V}$.

For any $u\in V^\bot$ and $v\in V$, using \eqref{delV} we have
\[ f_d(v)-\delta h(u) \leq f_d(v)-f_d(u)\leq f_d(u+v)\leq f_d(v)+f_d(u)\leq f_d(v)+\delta h(u)
\]
In particular, for $h(u)\leq\tau$, we get
\beq f_d(v)-\delta \tau \leq f_d(u+v)\leq f_d(v)+\delta \tau.
\label{in}\eeq
As required in assertion (ii), we have $\theta-\delta\tau\geq 0$.
If in addition $v\in V$ is such that $f_d(v)\leq\theta-\delta\tau$,
then $f_d(u+v)\leq\theta$.
Noticing that
$$\LS{f_d}{\theta}=\big\{u+v:(u,v)\in (V^\bot\!\!\times V),~f_d(u+v)\leq\theta\big\},$$
this implies that
\[B_0\defeq\big\{u+v:(u,v)\in (V^\bot\!\!\times V), ~h(u)\leq\tau,~f_d(v)\leq\theta-\delta\tau\big\}~
\subseteq ~V^\tau \cap \LS{f_d}{\theta}.
\]
Using that $f_d(u+v)\leq\theta$ (see the set we wish to measure in \eqref{encad}),
then the left-hand side of (\ref{in}) shows that
$f_d(v)\leq\theta+\delta\tau$, hence
\[B_1\defeq\big\{u+v:(u,v)\in (V^\bot\!\!\times V), ~h(u)\leq\tau,~f_d(v)\leq\theta+\delta\tau\big\}~
\supseteq ~V^\tau \cap \LS{f_d}{\theta}.
\]

Consider the pair of applications
\beqnn \varphi_0:\LS{h}{1}\times \left(V\cap\LS{f_d}{1}\right)&\to& \RR^N\\
(u,v)&\to&\tau u+(\theta-\delta\tau)v
\eeqnn
and
\beqnn \varphi_1:\LS{h}{1}\times \left(V\cap\LS{f_d}{1}\right)&\to& \RR^N\\
(u,v)&\to&\tau u+(\theta+\delta\tau)v
\eeqnn
Clearly, $\varphi_i$ is a Lipschitz homeomorphism satisfying
$\varphi_i\Big(\LS{h}{1}\times \big(V\cap\LS{f_d}{1}\big)\Big)=B_i$ for $i\in\{0,~1\}$.
Moreover, we have
\[D\varphi_0=\left[\barr{cc}\tau I_{N-K}&0\\ 0& (\theta-\delta\tau)I_K\earr\right]~~~\mbox{and}~~~
D\varphi_1=\left[\barr{cc}\tau I_{N-K}&0\\ 0& (\theta+\delta\tau)I_K\earr\right].\]
Then $\LEBEG {B_i}{N} $ can be computed using (see \cite{Evans92} for details)
\[\LEBEG {B_i}{N} =\int_{u\in\LS{h}{1}}\int_{v\in V\cap\LS{f_d}{1}}\JACOB{\varphi_i}dvdu,\]
where $\JACOB{\varphi_i}$ is the Jacobian of  $\varphi_i$, for $i=0$ or $i=1$.
In particular,
\beqnn \JACOB{\varphi_0}=\det \big( D\varphi_0\big)=\tau^{N-K}(\theta-\delta\tau)^K,\\
\JACOB{\varphi_1}=\det\big( D\varphi_1\big)=\tau^{N-K}(\theta+\delta\tau)^K.
\eeqnn
It follows that
\[\LEBEG {B_0}{N}=C\tau^{N-K}(\theta-\delta\tau)^K~~~\mbox{and}~~~
\LEBEG {B_1}{N}=C\tau^{N-K}(\theta+\delta\tau)^K\]
where the constant
\beqnn C&=&\int_{\LS{h}{1}}du\int_{V\cap\LS{f_d}{1}}dv\\
&=&\LEBEG{P_{V^\bot}\left(\LS{\|.\|}{1}\right)}{N-K} ~\LEBEG{V\cap\LS{f_d}{1}}{K}.
\eeqnn
Clearly $C$ is positive and finite.
Using the inclusion $B_0\subseteq ~V^\tau \cap \LS{f_d}{\theta}\subseteq B_1$ shows that
\[C\tau^{N-K}(\theta-\delta\tau)^K\leq
\LEBEG { V^\tau \cap \LS{f_d}{\theta}}{N}\leq C\tau^{N-K}(\theta+\delta\tau)^K.\]
The proof is complete.
\end{proof}

\subsection{Sets built from a dictionary}\label{CBD}

With every $J\subset I$, we associate the vector subspace $\TTJ{J}$ defined below:
\beq\TTJ{J}\defeq \Span\left((\psi_j)_{j\in J}\right),
\label{span}\eeq
along with the convention $\Span(\emptyset) = \{0\}$.
Given an arbitrary $\tau>0$, we introduce the subset of $\RR^N$
\beq\TTT{J}{\tau}\defeq \TTJ{J} + P_{\TOJ{J}}\left(\LS{\|.\|}{\tau}\right),
\label{TTJT}\eeq
where we recall that $\TOJ{J}$ is the orthogonal complement of
$\TTJ{J}$ in $\RR^N$ and $\|.\|$ is any norm on $\RR^N$.
These notations are constantly used in what follows.

The next assertion is a direct consequence of Proposition 1. The proposition is
illustrated on Figure \ref{dessin1}.
%(On Figure \ref{dessin1}, the notation $\ASTK{K}{\tau}$ is defined latter on, in \eqref{ASTK}.)

\DESSINUN

\begin{corollary}\label{rmkTJ}
For any $J\subset I$ (including $J=\emptyset$), any norm $\|.\|$ and any
$\tau>0$ the following hold:
\bit
\item[(i)] $\TTT{J}{\tau}$ is closed and measurable;

\item[(ii)] Let  $f_d$ be any norm on $\RR^N$ and $K\defeq\dim(\TTJ{J})$.
Then there exists $\delta_J\in[0,\odelta]$ (where $\odelta$ is given in Lemma \ref{lemme_h}(iii))
such that for $\theta\geq\delta_J\tau$ we have
\beq  C_J\tau^{N-K}(\theta-\delta_J\tau)^K~\leq~
\LEBEG { \TTT{J}{\tau}\cap \LS{f_d}{\theta}}{N}~\leq~ C_J\tau^{N-K}(\theta+\delta_J\tau)^K, \label{used}\eeq
where
\beq C_J=
\LEBEG{P_{\TTT{J}{\bot}}\left(\LS{\|.\|}{1}\right)}{N-K} ~ \LEBEG{\TTJ{J}\cap\LS{f_d}{1}}{K}
%\int_{P_{\TTT{J}{\bot}}\left(\LS{\|.\|}{1}\right)}du\int_{\TTJ{J}\cap\LS{f_d}{1}}dv
\in (0,+\infty).\label{used1}
\eeq
\eit
\end{corollary}

\begin{proof}
The corollary is a direct consequence of Proposition
\ref{propV}. Notice that we now write $\delta_J$ for the constant $\delta_\TTJ{J}$ in Lemma \ref{lemme_h}.

% JE N'AI PAS COMPRIS POURQUOI TU AVAIS BESOIN DE TOUT CA??
%The only thing to exhibit is the constant $\delta_J$.
%By (\ref{delV}), the constant $\delta_J$ in this corollary satisfies
%\beq f_d(u)\leq\delta_J h(u),~~~\forall u\in \TTT{J}{\bot},
%\label{delJ}\eeq
%where $h$ is the equivalent norm on $\TTT{J}{\bot}$ as exhibited in Lemma \ref{lemme_h}.
\end{proof}
It can be useful to remind that $\odelta$ is defined in
Lemma \ref{lemme_h} and only depends on $\|.\|$ and $f_d$.

A more friendly expression for $\TTT{J}{\tau}$ is provided by the
lemma below. Again, the lemma is illustrated on Figure \ref{dessin1}.

\begin{lemma}\label{+ball}
For any $J\subset I$ (including $J=\emptyset$), any norm $\|.\|$ and
$\tau>0$ let $\TTT{J}{\tau}$ be defined by (\ref{TTJT}). Then
$$\TTT{J}{\tau} = \TTJ{J} + \LS{\|.\|}{\tau}.$$
\end{lemma}

\begin{proof} The case $J=\emptyset$  is trivial  because of
  the convention $\Span(\emptyset) = \{0\}$. Consider next that $J$ is nonempty.
Let $w\in\TTT{J}{\tau}$, then $w$ admits a unique decomposition as
\[w=v+u~~~\mbox{where}~~ v\in\TTJ{J}~~\mbox{and}~~~u\in\TOJ{J}.\]
If $\|u\|\leq\tau$ then clearly $w\in\TTJ{J} + \LS{\|.\|}{\tau}$.
Consider next that $\|u\|>\tau$. From the definition of $\TTT{J}{\tau}$, there exists $w_u\in\LS{\|.\|}{\tau}$
such that $P_{\TOJ{J}}(w_u)=u$.
Noticing that $u-w_u = P_{\TOJ{J}}(w_u)-w_u\in\TTJ{J}$ and that $v+u-w_u\in\TTJ{J}$,
we can see that
\beqnn w&=&(v+u-w_u)+w_u \\ &\in&~~~~~~~ \TTJ{J}~~~~~~+ \LS{\|.\|}{\tau}.\eeqnn

Conversely, let $w\in\TTJ{J} + \LS{\|.\|}{\tau}$. Then
\[w=v_1+v~~~\mbox{where}~~~v_1\in\TTJ{J}~~\mbox{and}~~v\in\LS{\|.\|}{\tau}.\]
Furthermore, $v$ has a unique decomposition of the form
\[v=v_2+u~~~\mbox{where}~~ v_2\in\TTJ{J}~~\mbox{and}~~u\in\TOJ{J}.\]
In particular, \[u=P_{\TOJ{J}}(v)\in P_{\TOJ{J}}\left(\LS{\|.\|}{\tau}\right)\]
Combining this with the fact that $v_1+v_2\in\TTJ{J}$ shows that
$w=(v_1+v_2)+u\in\TTT{J}{\tau}$. %This concludes the proof of \eqref{firsteq}.
\end{proof}

\section{The intersection of two cylinder-like subsets is small}\label{intcyl}

\INTERSECTION

This section is devoted to prove quite an intuitive result on the estimate of the intersection of two
sets $\TTT{J}{\tau}$.
It uses all notations introduced in \S\ref{CBD} and is illustrated on Figure
\ref{intersection}.

\begin{proposition}\label{basic}
Let $J_1\subset I$ and $J_2\subset I$ be such that $\TTJ{J_1} \neq \TTJ{J_2}$ and
$\dim(\TTJ{J_1})=\dim(\TTJ{J_2})\defeq  K$.
Let $\tau>0$ and $\theta>0$.
Then the set given below
  \beq\TTT{J_1}{\tau} \cap \TTT{J_2}{\tau} \cap  \LS{f_d}{\theta}
 \label{TTTJ} \eeq
is closed and measurable.
Moreover, there is a constant ~$\delta_{\!J_1\!,\!J_2}\in[0,3\odelta]$
(where $\odelta$ is given in Lemma~\ref{lemme_h}(iii)) such that
for $\theta\geq \delta_{\!J_1\!,\!J_2}\tau$ we have
\[\LEBEG{\TTT{J_1}{\tau} \cap \TTT{J_2}{\tau} \cap \LS{f_d}{\theta}}{N}
    \leq Q_{\!J_1\!,\!J_2}\tau^{N-k}(\theta+\delta_{\!J_1\!,\!J_2}\tau)^k,~~~~
    k=\dim\big(\TTJ{J_1}\cap\TTJ{J_2}\big),
\]
where the constant $Q_{\!J_1\!,\!J_2}$ reads
\beq Q_{\!J_1\!,\!J_2}\defeq\LEBEG{W^\bot \cap \LS{\|.\|_2}{2\delta_2}}{N-k}~\LEBEG{W\cap  \LS{f_d}{1}}{k}
~~~~\mbox{for}~~~W\defeq\TTJ{J_1}\cap\TTJ{J_2}.\label{QJJ}\eeq
\end{proposition}

Notice that $Q_{\!J_1\!,\!J_2}$  depends only on $(\psi_j)_{j\in J_1}$
and $(\psi_j)_{j\in J_2}$, and the norms $\|.\|$ and $f_d$. A tighter
bound can be found in the proof of the proposition (see equation \eqref{Q'}). The bound is expressed in terms of a norm
$\tilde g$ constructed there.

\begin{remark}\label{nule}
Since $k=\dim W\leq K-1$, we have the following asymptotical result:
\beqnn\LEBEG{\TTT{J_1}{\tau} \cap \TTT{J_2}{\tau} \cap \LS{f_d}{\theta}}{N}
&\leq &Q_{\!J_1\!,\!J_2} ~\theta^N \left(\frac{\tau}{\theta}\right)^{N-k}
\left(1+\delta_{\!J_1,J_2} \frac{\tau}{\theta}\right)^{k} \\
&=&Q_{\!J_1\!,\!J_2} ~\theta^N \left(\frac{\tau}{\theta}\right)^{N-k}+
o\left(\left(\frac{\tau}{\theta}\right)^{N-k}\right)~~~~\mbox{as}~~\frac{\tau}{\theta}\to
0 \\
&=&\theta^N o\left(\left(\frac{\tau}{\theta}\right)^{N-K}\right)~~~~~~~~~~~~~~~~~~~~~~~~~\mbox{as}~~\frac{\tau}{\theta}\to
0.\eeqnn
\end{remark}

\begin{proof}
The subset in (\ref{TTTJ}) is closed and measurable, as being a
finite intersection of closed measurable sets.

Let
\[h_1:\TTT{J_1}{\bot}\to\RR~~~~\mbox{and}~~~~h_2:\TTT{J_2}{\bot}\to\RR\]
be the norms exhibited in Lemma \ref{lemme_h}---see equation \eqref{hu}---such that  for any $\tau\geq 0$,
\[\LS{h_1}{\tau}=P_{\TTJ{J_1^\bot}}\left(\LS{\|.\|}{\tau}\right)~~~~\mbox{and}~~~~
\LS{h_2}{\tau}=P_{\TTJ{J_2^\bot}}\left(\LS{\|.\|}{\tau}\right).\]
Reminding that by definition
\[W=\TTJ{J_1}\cap\TTJ{J_2},\]
De Morgan's law shows that
\[W^\bot=\TTT{J_1}{\bot}+\TTT{J_2}{\bot}.\]
Below we express the latter sum as a direct sum of subspaces:
\beqn W^\bot=\big(\TTT{J_1}{\bot}\cap\TTT{J_2}{\bot}\big)
&\oplus& \Big(\TTT{J_1}{\bot}\cap \TTJ{J_2} \Big)\label{AZ}\\
&\oplus& \Big(\;\TTJ{J_1}\;\cap \TTT{J_2}{\bot} \Big).\nonumber\eeqn
Notice that we have
\beq \barr{l}\TTT{J_1}{\bot}=\big(\TTT{J_1}{\bot}\cap\TTT{J_2}{\bot}\big)\oplus
\Big(\TTT{J_1}{\bot}\cap \TTJ{J_2}\Big),\\~\\
\TTT{J_2}{\bot}=\big(\TTT{J_1}{\bot}\cap\TTT{J_2}{\bot}\big)\oplus
\Big(\TTJ{J_1}\cap \TTT{J_2}{\bot}\Big),
\earr\label{a}\eeq
as well as
\beq\barr{l}\TTJ{J_1} = W \oplus \Big(\TTJ{J_1}\cap \TTT{J_2}{\bot}\Big),\\~
\\
\TTJ{J_2} = W \oplus \Big(\TTT{J_1}{\bot}\cap \TTJ{J_2}\Big).
\earr\label{b}\eeq
>From (\ref{AZ}), any $u\in W^\bot$ has a unique decomposition as
\beq u=u_1+u_2+u_3~~~~~~\mbox{where}~~~~~~ \barr{l}u_1\in \TTT{J_1}{\bot}\cap\TTT{J_2}{\bot}\\
                    u_2\in\TTT{J_1}{\bot} \cap \TTJ{J_2}\\
                    u_3\in\TTJ{J_1} ~\cap \TTT{J_2}{\bot}\earr
\label{decomp}\eeq

Using these notations, we introduce the following function:
\beqn g:W^\bot&\to&\RR \nonumber\\
u&\to& g(u)=\sup\big\{h_1(u_1+u_2),h_2(u_1+u_3)\big\}\label{g}.
\eeqn
In the next lines we show that $g$ is a norm on $W^\bot$:
\bit
\item $h_1$ and $h_2$ being norms, $g(\lambda u)=|\lambda| g(u)$, for all $\lambda\in\RR$;
\item if $g(u)=0$ then $u_1+u_2= u_1+u_3= 0$;
noticing that $u_1\bot u_2$ and that $u_1\bot u_3$ yields $u=0$;
\item for $u\in W^\bot$ and $v\in W^\bot$ (both decomposed according to (\ref{decomp})),
\beqnn g(u+v)&=&\sup\big\{h_1(u_1+u_2+v_1+v_2),h_2(u_1+u_3+v_1+v_3)\big\}\\
&\leq&\sup\big\{h_1(u_1+u_2)+h_1(v_1+v_2),h_2(u_1+u_3)+h_2(v_1+v_3)\big\}\\
&\leq&\sup\big\{h_1(u_1+u_2),h_2(u_1+u_3)\big\}+\sup\big\{h_1(v_1+v_2),h_2(v_1+v_3)\big\}\\
&=&g(u)+g(v).
\eeqnn
\eit
Furthermore, $g$ can be extended to a norm $\tilde g$ on $\RR^N$ such that
$\forall u\in W^\bot$, we have $\tilde g(u)=g(u)$ and
\begin{equation}\label{pjimied}
\LS{g}{\tau}=P_{W^\bot}\left(\LS{\tilde g}{\tau}\right),~~~~\forall \tau>0.
\end{equation}
Let us then define
\beqn  W^\tau &=&  W+ P_{W^\bot}\left(\LS{\tilde g}{\tau}\right) \nonumber\\
&=&\Big\{w+u:(u,w)\in ( W^\bot\!\!\times W),~g(u)\leq\tau \Big\}.\label{Wtau}
\eeqn
We are going to show that $\left(\TTT{J_1}{\tau} \cap\TTT{J_2}{\tau}\right)\subset W^\tau$.
In order to do so, we consider an arbitrary
\beq v\in\TTT{J_1}{\tau} \cap \TTT{J_2}{\tau}.\label{hawai}\eeq
It admits a unique decomposition of the form
\[v=w+u_1+u_2+u_3,\]
where $w\in W$, and $u_1, ~u_2$ and $u_3$ are decomposed according to (\ref{decomp}).
The latter, combined with \eqref{a} and \eqref{b} shows that
\beqnn
u_1+u_2\in\TTT{J_1}{\bot}&~~\mbox{and}~~&w+u_3\in\TTJ{J_1},\\
u_1+u_3\in\TTT{J_2}{\bot}&~~\mbox{and}~~&w+u_2\in\TTJ{J_2}.
\eeqnn
The inclusions given above, combined with \eqref{hawai}, show that
\[ h_1(u_1+u_2)\leq\tau~~~~~\mbox{and}~~~~~h_2(u_1+u_3)\leq\tau.\]
By the definition of $g$ in \eqref{decomp}-(\ref{g}), the inequalities given above imply  that $g(u)\leq\tau$.
Combining this with the definition of $W^\tau$ in (\ref{Wtau}) entails that $v\in W^\tau$.
Consequently,
\[\Big(\TTT{J_1}{\tau} \cap \TTT{J_2}{\tau}\Big)\subset W^\tau~~~~\mbox{and}~~~~
\Big(\TTT{J_1}{\tau} \cap \TTT{J_2}{\tau}\cap \LS{f_d}{\theta}\Big)~\subset~ \Big(W^\tau\cap \LS{f_d}{\theta}\Big). \]
It follows that
\[\LEBEG{\TTT{J_1}{\tau} \cap \TTT{J_2}{\tau}\cap \LS{f_d}{\theta}}{N}
            \leq \LEBEG{W^\tau \cap \LS{f_d}{\theta}}{N}.\]
Applying now the right-hand side of (\ref{encad}) in
Proposition \ref{propV} with $W^\tau$ in place
of $V^\tau$ and taking $\delta_{\!J_1\!,\!J_2}$ such that
\beq f_d(u)\leq\delta_{\!J_1\!,\!J_2} g(u),~~~\forall u\in W^\bot,\label{del12}\eeq
leads to
\[\LEBEG{W^\tau \cap \LS{f_d}{\theta}}{N}\leq Q'_{\!J_1\!,\!J_2}\tau^{N-k}(\theta+\delta_{\!J_1\!,\!J_2}\tau)^k,\]
where it is easy to see that
\begin{equation}\label{Q'}
Q'_{\!J_1\!,\!J_2}=\LEBEG{P_{W^\bot}\left(\LS{\tilde g}{\tau}\right)}{N-k} \LEBEG{W\cap
  \LS{f_d}{1}}{k}=\LEBEG{\LS{g}{1}}{N-k} \LEBEG{W\cap
  \LS{f_d}{1}}{k}.
\end{equation}
In order to obtain (\ref{QJJ}), we are going to show that
$\LS{g}{1}\subset \left(\LS{\|.\|_2}{2\delta_2}\cap W^\bot\right)$. Using Lemma \ref{lemme_h}
(ii), if $u \in W^\bot$ is decomposed according to \eqref{decomp}, we obtain
\begin{eqnarray}
\|u\|_2 = \left(\|u_1\|_2^2+\|u_2\|_2^2+\|u_3\|_2^2 \right)^{\frac{1}{2}}& \leq &  \|2 u_1+u_2+u_3\|_2 \nonumber\\
& \leq &  \|u_1+u_2\|_2 + \|u_1+u_3\|_2 \nonumber\\
&\leq & \delta_2 h_1(u_1+u_2) + \delta_2 h_2(u_1+u_3)\nonumber\\
&\leq & 2 \delta_2 g(u).\nonumber
\end{eqnarray}
So $\LS{g}{1}\subset \left(\LS{\|.\|_2}{2\delta_2}\cap W^\bot\right)$ and
$Q'_{\!J_1\!,\!J_2}\leq Q_{\!J_1\!,\!J_2}$, for $Q_{\!J_1\!,\!J_2}$ as
given in the proposition.

At last, we need to build a uniform bound on
$\delta_{\!J_1\!,\!J_2}$ giving rise to \eqref{del12}.
Using Lemma \ref{lemme_h} (ii), if $u \in W^\bot$ is decomposed according to \eqref{decomp},
we obtain
\begin{eqnarray}
f_d(u) = f_d(u_1+u_2+u_3) & \leq &  f_d(2 u_1+u_2+u_3) + f_d(u_1) \nonumber\\
& \leq &  f_d(u_1+u_2) + f_d(u_1+u_3) +  f_d(u_1) \nonumber\\
&\leq & \odelta h_1(u_1+u_2) + \odelta h_2(u_1+u_3) +\delta_1
\|u_1\|_2 .\label{ourbn}
\end{eqnarray}
Using \eqref{major2}, $\|u_1\|_2$ satisfies the  following two inequalities
\beqnn\|u_1\|_2 &\leq &\|u_1+u_2\|_2 \leq \delta_2 h_1(u_1+u_2),
\\
\|u_1\|_2 &\leq &\|u_1+u_3\|_2 \leq \delta_2 h_2(u_1+u_3).
\eeqnn
Adding these inequalities, we obtain
\[\delta_1 \|u_1\|_2 \leq \frac{\odelta}{2} ~ \big(h_1(u_1+u_2)+h_2(u_1+u_3)\big).
\]
Using \eqref{ourbn}, we finally conclude that, for $u \in W^\bot$
\begin{eqnarray*}
f_d(u) & \leq & \frac{3\odelta}{2} ~  \big( h_1(u_1+u_2)+h_2(u_1+u_3)\big)\\
 & \leq & 3 \odelta ~ g(u).
\end{eqnarray*}
The proof is complete.
\end{proof}

\section{Sets of data yielding $K$-sparse solutions or sparser}\label{sparser}

For any given $K\in\{0,\ldots, N\}$ and $\tau>0$, we introduce the subset $\ASTK{K}{\tau}$ as it follows:
\beqn\ASTK{K}{\tau}  \defeq \big\{d\in \RR^N~:~ \val(\P_d)\leq K \big\} \label{ASTK}.
\eeqn
All data belonging to $\ASTK{K}{\tau}$ generate a solution of $(\P_d)$---see (\ref{Pd})---which involves at most
$K$ non-zero components.

Let us define
\begin{equation}
G_K \defeq \big\{J\subset I: \dim(\TTJ{J})\leq K\big\},\label{GK}
\end{equation}
and remind that $\TTJ{J}=\Span\left((\psi_j)_{j\in J}\right)$ according to  (\ref{span}).

The next proposition states a strong and slightly surprising result.
\begin{proposition}\label{astk=sstk}
For any $K\in\{0,\ldots, N\}$, any norm $\|.\|$ and any $\tau>0$, we have
\[\ASTK{K}{\tau} = \bigcup_{J\in G_K} \TTJ{J} + \LS{\|.\|}{\tau}.
\]
\end{proposition}

Some sets $\ASTK{K}{\tau}$, as defined in \eqref{ASTK} and explained in the last proposition, are illustrated on
Figure \ref{dessin1}.

\begin{proof}
The case $K=0$ is trivial ($G_0=\{\emptyset\}$) and we assume in the following that $K\geq 1$.

Let $d\in \ASTK{K}{\tau}$. This means there is $(\lambda_i)_{i\in I}$---a
solution of $(\P_d)$---that satisfies $\ell_0((\lambda_i)_{i\in I})\leq K$.
Hence
\beqnn d=\sum_{i\in J} \lambda_i \psi_i + w~&&\mbox{with}~~~w\in\LS{\|.\|}{\tau}\\
&&\mbox{and}~~~J=\{i\in I : \lambda_i\neq 0\}~~~\mbox{with}~~~ \# J\leq K.
\eeqnn
Consequently $\dim(\TTJ{J})\leq\# J\leq K$, which implies that
$d\in \cup_{J\in G_K} \TTJ{J} + \LS{\|.\|}{\tau}$.

Conversely, let $d\in\cup_{J\in G_K} \TTJ{J} + \LS{\|.\|}{\tau}$, then
$d=v+w$ where $v \in \cup_{J\in G_K} \TTJ{J} $ and $w\in\LS{\|.\|}{\tau}$.
Then:
\bit
\item $\exists~J\subset I$ such that $v\in \TTJ{J}$ and the latter satisfies $\dim(\TTJ{J})\leq K$;
\item there are real numbers $(\lambda_i)_{i\in J}$ involving at most
  $\dim(\TTJ{J})$ non-zero components (hence $\ell_0((\lambda_i)_{i\in
    J})\leq \dim(\TTJ{J})\leq K$) such that $v=\sum_{i\in J} \lambda_i\psi_i$.
\item  $w\in\LS{\|.\|}{\tau}$ means that $\|w\|\leq\tau$.
\eit
It follows that ~$d=\sum_{i\in J}\lambda_i\psi_i+w ~\in\ASTK{K}{\tau}$.
\end{proof}

Given $J\subset I$, remind that $\TTJ{J}= \Span\left((\psi_j)_{j\in J}\right)$---see (\ref{span}).
Since $(\psi_i)_{i\in I}$ is a general family of vectors, there may be numerous subsets
$J_n$, $n=1,2,\ldots$, such that $\TTJ{J_n}=\TTJ{J_m}$ and $J_n\neq J_m$.
A non-redundant listing of all possible subspaces $\TTJ{J}$ when $J$ runs over all subsets of $I$ can be obtained
with the help of the notations below.

For any $K=0,\ldots,N$, define $\J(K)$ by the following three properties:
\beq \left\{\barr{lll}
(a)&&\J(K) \subset \big\{J\subset I : \dim(\TTJ{J})=K \big\};\\~\\
(b)&&  J_1,J_2\in \J(K)~~\mbox{and} ~~J_1\neq J_2 ~~\Longrightarrow~~ \TTJ{J_1}\neq \TTJ{J_2};\\~\\
(c)&&\J(K) ~\mbox{is maximal:}\\
&&\mbox{if~~} J_1\subset I~~\mbox{yields}~~ \dim(\TTJ{J_1})=K~~\mbox{then}~~\exists J\in\J(K)
~~\mbox{such that}~~\TTJ{J}=\TTJ{J_1}.
\earr\right.\label{Jk}\eeq
Notice that in particular, $\J(0) = \{\emptyset\}$ and $\#\J(N) = 1$.
One can observe that $G_K$, as defined in (\ref{GK}), satisfies
\[ G_K\supset\bigcup_{k=0}^K\J(k)
\]
and
\begin{equation}\label{GK=Jk}
\{ \TTJ{J} : J\in G_K\} = \{ \TTJ{J} : J\in \J(k) \mbox{ for } k\in\{0,\ldots,K\}\}.
\end{equation}

Using these notations, we can give a more convenient formulation of Proposition \ref{astk=sstk}.

\begin{theorem}\label{proj=all}
For any $K\in\{0,\ldots, N\}$, any norm $\|.\|$ and any $\tau>0$, we have
\[\ASTK{K}{\tau} = \bigcup_{J\in \J(K)} \TTT{J}{\tau},
\]
where we remind that for any $J\subset I$ and $\tau>0$,
$\TTT{J}{\tau}$ is defined by (\ref{TTJT}), and  $\J(K)$ is defined by (\ref{Jk}).

As a consequence, $\ASTK{K}{\tau}$ is closed and measurable.
\end{theorem}

\begin{proof} The case $J=\emptyset$ (and $K=0$) is trivial because of
  the convention $\Span(\emptyset) = \{0\}$ and $\J(0)=\{\emptyset\}$.

Let us first prove that $\ASTK{K}{\tau}=\bigcup_{J\in G_K}
\TTT{J}{\tau}$. Using Proposition \ref{astk=sstk},
\[\ASTK{K}{\tau}= \left(\bigcup_{J\in G_K}\TTJ{J}\right) + \LS{\|.\|}{\tau}=
\bigcup_{J\in G_K}\big(\TTJ{J} + \LS{\|.\|}{\tau}\big).
\]
The  last equality above is a trivial observation. Using Lemma \ref{+ball},
this summarizes us
\[\ASTK{K}{\tau}=\bigcup_{J\in G_K} \TTT{J}{\tau}.
\]
Using \eqref{GK=Jk}, we deduce that
\[\{ \TTT{J}{\tau} : J\in G_K\} = \big\{ \TTT{J}{\tau} : J\in \J(k) \mbox{ for } k\in\{0,\ldots,K\}\big\},
\]
and therefore,
\[ \ASTK{K}{\tau}=\bigcup_{k=0}^K\bigcup_{J\in \J(k)}\TTT{J}{\tau}.\]
Moreover, for any $k<K$ and $J\in \J(k)$, we can find
$J_1\in \J(K)$ such that $\TTJ{J}\subset\TTJ{J_1}$.
Using Lemma \ref{+ball}, we find that $\TTT{J}{\tau}\subset \TTT{J_1}{\tau}$.
Consequently,
\[ \ASTK{K}{\tau}= \bigcup_{J\in \J(K)}\TTT{J}{\tau}.
\]
This completes the proof of the first statement.

By Proposition \ref{propV}, $\ASTK{K}{\tau}$ is a finite union of
closed measurable sets, hence it is closed and measurable as well.
\end{proof}

For any $K=0,\ldots,N$, define the constants $\hat\delta_K$ and $\CK{K}$ as it follows:
\beqn \hat\delta_K&\defeq&\max_{J\in\J(K)}\delta_J,\label{deltaK}\\
\CK{K}  &\defeq& \sum_{J\in {\mathcal J}(K)} C_J,\label{CK}
\eeqn
where $\delta_J\in[0,\odelta]$ and $C_J$ are the constants exhibited in Corollary \ref{rmkTJ}, assertion (ii).
Clearly,
\begin{equation} 0\leq\hat\delta_K\leq\odelta.\label{Q1} \end{equation}
In particular,
\begin{equation}  \CK{0} = \LEBEG{\LS{\|.\|}{1}}{N}\mbox{ and }\CK{N} = \LEBEG{\LS{f_d}{1}}{N}.\label{CNC0}
\end{equation}

With  $\J(K)$, let us associate the family of subsets :
\beq \H(K,k)\defeq\Big\{(J_1,J_2)\in\J(K)^2~~\mbox{such that}~~
\dim(\TTJ{J_1} \cap \TTJ{J_2})=k\Big\},\label{Hk}\eeq
where $K=1,2,\ldots,N$ and $k=0,1,\ldots,K-1$.

Notice that $\H(K,k)$ may be empty for some $k$.
%In particular, if
%$\H(K,k)\neq \emptyset$ there exists
Consider  $(J_1,J_2)\in\J(K)^2$ such that
\begin{eqnarray*}
\TTJ{J_1} + \TTJ{J_2} = (\TTJ{J_1}\cap \TTJ{J_2}) \oplus
(\TTJ{J_1}\cap \TTJ{J_2}^\bot)\oplus (\TTJ{J_2}\cap \TTJ{J_1}^\bot)
&\subset & \RR^N \\
\dim(\TTJ{J_1} + \TTJ{J_2}) =~~~~~~k ~~~~~~ + ~~~ (K-k)  ~~~  + ~~~ (K-k) ~~~& \leq & N
\end{eqnarray*}
and $k\geq 2K-N$. We see that $$\H(K,k)\neq \emptyset ~~~~\Rightarrow~~~~ k \geq 2K-N.$$
Conversely,
\beq k < k_K\defeq\max\{0,~2K-N\} ~~~~\Rightarrow ~~~~\H(K,k)=\emptyset.
\label{ko}\eeq
Notice that $\H(N,k) = \emptyset$, for all $k=0,\ldots, N-1$ and that
for any $K=1,\ldots,N-1$, we have $0\leq k_K \leq K-1$.

For $K\in\{1,\ldots,N-1\}$ and $k\in\{k_K,\ldots,K-1\}$ let us define
\beqn
\hat\delta'_{K,k}&\defeq& \disp{\max\left\{0, \max_{(J_1, J_2)\in\H(K,k)}\delta_{J_1,J_2}\right\}},\label{dpK}\\
\QK{k}&\defeq& \sum_{(J_1, J_2)\in\H(K,k)}
Q_{\!J_1\!,\!J_2}\label{QK}
\eeqn
where $Q_{\!J_1\!,\!J_2}$ and $\delta_{J_1, J_2}\in[0,3\odelta]$ are as in Proposition \ref{basic}.
It is clear  that  if  $\H(K,k)=\emptyset$ then we  find
$\QK{k}=0$ and $\hat\delta'_{K,k}=0$.
It follows  that for any  $K=1, \ldots, N-1$ and any $k = k_K,\ldots,K-1$
\beq 0\leq\hat\delta'_{K,k}\leq 3\odelta.\label{Q2}\eeq
Last, define
\beq\Delta_K\defeq\left\{\begin{array}{ll}
\hat\delta_0 & \mbox{, if } K=0 \\
\disp{\max\left\{ \Delta_{K-1}, \hat\delta_K,\max_{k_K\leq k\leq K-1}\hat\delta_{K,k}'\right\}} & \mbox{, if } 0<K<N \\
\max\big\{\Delta_{N-1}, \hat\delta_N\big\} & \mbox{, if } K=N \\
\end{array}\right.\label{delmax}\eeq
Using \eqref{Q1} and \eqref{Q2},
\beq 0\leq \Delta_K\leq 3\odelta.\label{Q3}\eeq

All these constants, introduced between \eqref{deltaK} and \eqref{delmax}, depend only on the family $(\psi_i)_{i\in I}$, the
norms $\|.\|$ and $f_d$, $K$ and $k$. Their upper bounds using $\odelta$
only depend on $\|.\|$ and $f_d$. They are involved in the theorem
below which provides a critical result in this work.

\begin{theorem}\label{propSSTK}
Let $K\in\{0,\ldots, N\}$, the norms $\|.\|$ and  $f_d$,
and $(\psi_i)_{i\in I}$, be any.
Let $\tau>0$ and $\theta \geq \tau \Delta_K$ where
$\Delta_K$ is defined in (\ref{delmax}). The Lebesgue measure in $\RR^N$ of the set $\ASTK{K}{\tau}$
defined by (\ref{ASTK}) satisfies
\beq\CK{K} \tau^{N-K} (\theta-\hat\delta_K \tau)^K - \theta^N~\eps_0(K,\tau,\theta) ~\leq~
\LEBEG{\ASTK{K}{\tau}\cap \LS{f_d}{\theta}}{N}
~\leq ~\CK{K}\tau^{N-K} (\theta+\hat\delta_K \tau)^K ,
\label{R1}\eeq
where
\beq \eps_0(K,\tau,\theta) = \left\{\begin{array}{ll}
~~~~0 & \mbox{\rm if } K = 0 \mbox{~~\rm  or } K=N\\
\disp{\sum_{k=k_K}^{K-1}\QK{k}\left(\frac{\tau}{\theta}\right)^{N-k}\Big(1+\hat\delta_{K,k}' \frac{\tau}{\theta}\Big)^k}
& \mbox{\rm  if } 0< K < N\\
\end{array}\right.
\label{eps_0}\eeq
for $\CK{K}$, $k_K$, $\QK{k}$, $\hat\delta_k$ and  $\hat\delta_{K,k}'$
defined by \eqref{CK}, \eqref{ko}, \eqref{QK}, \eqref{deltaK} and
 \eqref{dpK} respectively.
 Moreover, \eqref{Q1},  \eqref{Q2} and \eqref{Q3} provide bounds on
$\hat\delta_K$, $\hat\delta_{K,k}'$ and  $\Delta_K$,
respectively, which depend only on  $\|.\|$ and $f_d$, via
$\odelta$ (see Lemma \ref{lemme_h}  (iii)).
\end{theorem}

\begin{remark}\label{rmkASTK} We posit the assumptions of Theorem
  \ref{propSSTK}. Then asymptotically
\[\LEBEG{\ASTK{K}{\tau} \cap \LS{f_d}{\theta}}{N}
= \CK{K} ~\theta^{N}\left(\frac{\tau}{\theta}\right)^{N-K} +\theta^N ~o\left(
  \left(\frac{\tau}{\theta}\right)^{N-K}
\right)~~\mbox{as}~~\frac{\tau}{\theta}\to 0.\]
\end{remark}

\begin{proof}
Using Theorem \ref{proj=all}, it is straightforward that
\beq
\ASTK{K}{\tau}\cap \LS{f_d}{\theta}
~=~ \bigcup_{J\in \J(K)}\Big(\TTT{J}{\tau}\cap \LS{f_d}{\theta}\Big)
\label{EQUA}
\eeq
and that
\beq
\LEBEG{\ASTK{K}{\tau}\cap \LS{f_d}{\theta}}{N} =
\LEBEG{\bigcup_{J\in\J(K)}\Big(\TTT{J}{\tau}\cap \LS{f_d}{\theta}\Big)}{N}.
\label{evi}
\eeq
When $K=0$ or $K=N$, we have $\#\J(K)=1$. Then, \eqref{R1} is a
straightforward consequence of \eqref{evi} and Proposition \ref{propV}
(the latter can be applied thanks to the assumption $\theta > \tau
\Delta_K$ and \eqref{delmax}).

The rest of the proof is to find relevant bounds for the right-hand
side of (\ref{evi}) under the assumption that $0<K<N$.

\paragraph{\it Upper bound.}
By the definition of a measure, and then using
Corollary \ref{rmkTJ}, it is found that
\beqn \LEBEG{\ASTK{K}{\tau}\cap \LS{f_d}{\theta}}{N}
&\leq&
\sum_{J\in\J(K)}\LEBEG{\TTT{J}{\tau}\cap \LS{f_d}{\theta}}{N} \label{missing}\\
&\leq &
\tau^{N-K}\sum_{J\in\J(K)} C_J(\theta+\delta_J\tau)^K  \nonumber\\
&\leq &
\tau^{N-K}(\theta+\hat\delta_K\tau)^K\sum_{J\in\J(K)} C_J \nonumber\\
&= & \CK{K}~\tau^{N-K}(\theta+\hat\delta_K\tau)^K ,\nonumber
\eeqn
where the constants $\hat\delta_K$ and $\CK{K}$  are defined in (\ref{deltaK}) and (\ref{CK}),
respectively.

\paragraph{\it Lower bound.}
First we represent
the right-hand side of (\ref{EQUA}) as a union of disjoint subsets.
Since $\J(K)$ is finite, let us enumerate its elements as
\[\J(K)=\{J_1,\ldots,J_M\}~~\mbox{where}~~M=\#\big(\J(K)\big).
\]
To simplify the expressions that follow, for any $J$ we denote
\beq B_J=\TTT{J}{\tau}\cap \LS{f_d}{\theta}\label{BJ}.\eeq
Then
\[\bigcup_{J\in\J(K)}\Big(\TTT{J}{\tau}\cap \LS{f_d}{\theta}\Big)=
\bigcup_{i=1}^M B_{J_i}.\]
Consider the following decomposition:
\beqnn \bigcup_{i=1}^M B_{J_i}&=&
\Big(B_{J_1}\Big)\cup \Big(B_{J_2}\setminus(B_{J_1}\cap B_{J_2})\Big)\cup\ldots\cup
\Big(B_{J_M}\setminus\left(\cup_{j=1}^{M-1} (B_{J_j} \cap B_{J_M})\right)\Big)\\
&=&\Big(B_{J_1}\Big)\cup~
\bigcup_{i=2}^M\left(B_{J_i}\setminus\Big(\bigcup_{j=1}^{i-1} \big(B_{J_j}
  \cap B_{J_i} \big) \Big)\right).
\eeqnn
Since the last row is a union of disjoint sets, we have
\[\LEBEG{\bigcup_{i=1}^M B_{J_i}}{N} = \LEBEG{B_{J_1}}{N} + \sum_{i=2}^M \LEBEG{(B_{J_i}\setminus\left(\cup_{j=1}^{i-1} (B_{J_j}
  \cap B_{J_i} ) \right)}{N}.
\]
Noticing that $\left(\bigcup_{j=1}^{i-1} (B_{J_j} \cap B_{J_i})\right)\subset B_{J_i}$ entails that
\[\LEBEG{B_{J_i}\setminus\big(\cup_{j=1}^{i-1} (B_{J_j}\cap B_{J_i})\big)}{N}
  =\LEBEG{B_{J_i}}{N}-\LEBEG{\cup_{j=1}^{i-1} (B_{J_j}\cap B_{J_i})}{N},~~~\forall i=2,\ldots,M.\]
Hence
\beq\LEBEG{\bigcup_{i=1}^M
  B_{J_i}}{N}=\sum_{i=1}^M\LEBEG{B_{J_i}}{N}-\sum_{i=2}^M\LEBEG{\bigcup_{j=1}^{i-1}(B_{J_j}\cap B_{J_i})}{N}.
\label{mlk}\eeq

Using  successively \eqref{BJ}, assertion (ii) of Corollary
\ref{rmkTJ}, \eqref{deltaK}, \eqref{CK} and $\theta \geq \tau\Delta_K$
shows that
\beqn
\sum_{i=1}^M\LEBEG{B_{J_i}}{N} &=& \sum_{J\in \J(K)}\LEBEG{\TTT{J}{\tau}\cap \LS{f_d}{\theta}}{N}\nonumber \\
&\geq& \sum_{J\in \J(K)} C_J\tau^{N-K}(\theta-\delta_J\tau)^K \nonumber\\
&\geq& \CK{K}\tau^{N-K} (\theta - \hat\delta_K \tau)^K,\label{pwjrv}
\eeqn
where the constants $\hat\delta_K$ and $\CK{K}$ are given in (\ref{deltaK}) and (\ref{CK}), respectively.

Using the original notation (\ref{BJ}), each term, for $i=2,...,M$, in the last sum in \eqref{mlk} satisfies
\beq\LEBEG{\bigcup_{j=1}^{i-1} (B_{J_j}\cap B_{J_i})}{N}
\leq \sum_{j=1}^{i-1}\LEBEG{B_{J_j}\cap B_{J_i}}{N}
=\sum_{j=1}^{i-1}
\LEBEG{\LS{f_d}{\theta}\cap \TTT{J_j}{\tau} \cap \TTT{J_i}{\tau}}{N}. \label{mnb}
\eeq
Let us remind that $\dim(\TTJ{J_i}) =K$ for every $i=1,\ldots,M$ and that by the definition of
$\J(K)$---see (\ref{Jk})---we have $\TTJ{J_j}\neq \TTJ{J_i}$ if $i\neq j$.
Proposition \ref{basic} can hence be applied to each term of the last sum:
\beqnn \LEBEG{\LS{f_d}{\theta}\cap \TTT{J_j}{\tau} \cap \TTT{J_i}{\tau}}{N}&\leq&
Q_{J_i,J_j}\tau^{N-k_{i,j}}(\theta +\delta_{_{J_i,J_j}} \tau)^{k_{i,j}}\\
\mbox{where}&& k_{i,j}=\dim\big(\TTJ{J_j}\cap \TTJ{J_i}\big).
\eeqnn
Then (\ref{mnb}) leads to
\[\LEBEG{\bigcup_{j=1}^{i-1} (B_{J_j}\cap B_{J_i})}{N} \leq
 \sum_{j=1}^{i-1}Q_{J_j,J_i}\tau^{N-k_{i,j}}(\theta +\delta_{_{J_j,J_i}} \tau)^{k_{i,j}}.
\]
By rearranging the last sum in (\ref{mlk}) and taking into account \eqref{ko}, we obtain
\beq\sum_{i=2}^{M} \LEBEG{\bigcup_{j=1}^{i-1} (B_{J_j}\cap   B_{J_i})}{N}
\leq  \sum_{k=k_K}^{K-1} {\QK{k}} \tau^{N-k} (\theta +\hat\delta'_{K,k} \tau)^k,
\label{fi}\eeq
where $\hat\delta'_{K,k}$ and $\QK{k}$ are given in (\ref{dpK}) and (\ref{QK}), respectively.

Combining (\ref{evi}) along with the original notations (\ref{BJ}) and then
(\ref{mlk}),  (\ref{pwjrv}) and (\ref{fi}) yields
\beqn \LEBEG{ \ASTK{K}{\tau}\cap \LS{f_d}{\theta} }{N}& = &
 \LEBEG{\bigcup_{i=1}^M B_{J_i}}{N}  \nonumber \\
&\geq& \CK{K}\tau^{N-K}(\theta-\hat\delta_K\tau)^K  - \eps_0(K,\tau,\theta)\nonumber,
\eeqn
where $\eps_0(.)$ is as in the proposition.
This finishes the proof.
\end{proof}

\begin{remark}
In the proof of this theorem we could notice (see \eqref{mlk}, \eqref{mnb} and \eqref{missing}) that

\beqn
\sum_{J\in \J(K)}\LEBEG{\TTT{J}{\tau}\cap \LS{f_d}{\theta}}{N}
&-&\sum_{k=k_K}^{K-1}\sum_{(J_1, J_2)\in\H(K,k)}\LEBEG{\LS{f_d}{\theta}\cap \TTT{J_1}{\tau} \cap \TTT{J_2}{\tau}}{N}
\nonumber\\
&\leq&
\LEBEG{\ASTK{K}{\tau}\cap \LS{f_d}{\theta}}{N} \label{whoa}\\
&\leq&\sum_{J\in\J(K)}\LEBEG{\TTT{J}{\tau}\cap \LS{f_d}{\theta}}{N}.\nonumber
\eeqn
These are the main approximations of $\LEBEG{\ASTK{K}{\tau}\cap
  \LS{f_d}{\theta}}{N}$ in the proof of the theorem. The
precision of the bounds given in the theorem could be more accurate by
improving the above inequalities. The loss of accuracy has however the
same order of magnitude as the precision in the calculus of
$\LEBEG{\TTT{J}{\tau}\cap \LS{f_d}{\theta}}{N}$.
\end{remark}

The constants $\Delta_K$, $\hat\delta_K$ and $\hat\delta_{K,k}'$ depend on
$(\psi_i)_{i\in I}$ and $K$.
Using the uniform
bound $\odelta$ exhibited in Lemma \ref{lemme_h} (ii) in place of $\hat\delta_K$ and $\hat\delta'_{K,k}$ leads
to a more general but less precise result.
\begin{corollary}\label{corUnifBounds}
Let $K\in\{0,\ldots, N\}$, the norms $\|.\|$ and  $f_d$,
and $(\psi_i)_{i\in I}$, be any.
Let $\tau>0$ and $\theta \geq 3 \tau \odelta$ where
$\odelta$ is derived in Lemma \ref{lemme_h} (ii) and depends only on $f_d$ and $\|.\|$.
The set $\ASTK{K}{\tau}$
defined by (\ref{ASTK}) satisfies
\begin{equation}\label{peubn}
\CK{K} \tau^{N-K} (\theta-\odelta \tau)^K - \theta^N~\eps^u_0(K,\tau,\theta) \leq
\LEBEG{\ASTK{K}{\tau}\cap \LS{f_d}{\theta}}{N}
\leq \CK{K}\tau^{N-K} (\theta+\odelta \tau)^K,
\end{equation}
where
\[\eps^u_0(K,\tau,\theta)=\left\{\begin{array}{ll}
0 & \mbox{, if } K=0 \mbox{ or }K=N \\
 \disp{\sum_{k=k_K}^{K-1} \QK{k}\left(\frac{\tau}{\theta}\right)^{N-k} \left(1 +
 3\odelta \,\frac{\tau}{\theta}\right)^k}& \mbox{, if } 0<K<N.\\
\end{array}\right.
\]
Moreover, for $K=1,\ldots,N-1$ and $k=k_K,\ldots,K-1$,  we have
\begin{equation}\label{QKub}
\QK{k} \leq \#\J(K) (\#\J(K) - 1) \alpha(N-k)\alpha(k)
 (2\delta_2)^{N-k} \delta_3^k
\end{equation}
where
\begin{equation}\label{JKub}
\#\J(K) \leq \COMB{K}{\# I},
\end{equation}
$\alpha(n)$ is the volume of unit ball for the euclidean norm in
$\RR^n$ (see equation \eqref{alpha} for details), $\delta_2$ is defined in Lemma \ref{lemme_h} (see equation \eqref{Y})
% and is such that
%\[\|w\|_2 \leq \delta_2 \|w\| ~~~~~, \forall w\in\RR^N,
%\]
and $\delta_3$ is such that
\[\|w\|_2 \leq \delta_3 f_d(w)~,  ~~~~\forall w\in\RR^N.
\]
\end{corollary}
\begin{proof} Equation \eqref{peubn} is obtained by inserting in \eqref{R1} in Theorem \ref{propSSTK} the uniform bounds on
$\hat\delta_K$, $\hat\delta_{K,k}'$ and  $\Delta_K$ given in
\eqref{Q1},  \eqref{Q2} and \eqref{Q3}, respectively.

The upper bound for $\QK{k}$ is obtained as follows. Using
\eqref{QK} and \eqref{QJJ}, we obtain
\[\QK{k} = \sum_{(J_1, J_2)\in\H(K,k)} \LEBEG{(\TTJ{J_1}\cap\TTJ{J_2}
  )^\bot \cap
  \LS{\|.\|_2}{2\delta_2}}{N-k}~\LEBEG{\TTJ{J_1}\cap\TTJ{J_2}\cap
  \LS{f_d}{1}}{k}.
\]
Moreover,
\[\LEBEG{(\TTJ{J_1}\cap\TTJ{J_2}
  )^\bot \cap\LS{\|.\|_2}{2\delta_2}}{N-k} = \alpha(N-k) (2\delta_2)^{N-k},
\]
\[\LEBEG{\TTJ{J_1}\cap\TTJ{J_2}\cap
  \LS{f_d}{1}}{k} \leq \LEBEG{\TTJ{J_1}\cap\TTJ{J_2}\cap
  \LS{\|.\|_2}{\delta_3}}{k} = \alpha(k) (\delta_3)^{k},
\]
and we obviously have
\[\#\H(K,k) \leq \#\J(K) (\#\J(K)-1).
\]
\end{proof}
The above corollary shows that the
``quality'' of the asymptotic as
$\frac{\tau}{\theta}\rightarrow 0$ depends on $\|.\|$, $f_d$ and
on the dictionary through the terms $\QK{k}$. The latter terms are
bounded from above using \eqref{QKub} and \eqref{JKub} and they are clearly overestimated.
Even though the bound we provide are very pessimistic, they depend only on $\|.\|$,
$f_d$ and $\#I$ and can be computed.

\begin{remark}
%The above corollary is a straightforward adaptation of Theorem \ref{propSSTK}.
Let us emphasize that ``uniform'' bounds in the spirit of Corollary \ref{corUnifBounds} can be
derived from  Proposition \ref{propSSTKproba}, and Theorems \ref{corASTKE},
\ref{thm=proba} and \ref{thmexpect}.
We leave this task to interested readers that need to compute easily the relevant bounds.
%It is possible to adapt Theorem \ref{corASTKE}, Proposition \ref{propSSTKproba}, Theorem
%\ref{thm=proba} and Theorem \ref{thmexpect} in a similar way, in order
%to get easy to compute bounds. We will not do it to avoid the repetition
%of those results.
\end{remark}
\section{Sets of data yielding $K$-sparse solutions}

For any $K\in\{0,\ldots, N\}$ and $\tau>0$, we denote
\beq\ASTKE{K}{\tau}  \defeq \left\{d\in \RR^N: \val(\P_d) = K \right\}.
\label{DTK}\eeq
>From the definition of $\ASTK{K}{\tau}$ in (\ref{ASTK}), it is straightforward that
\begin{equation}\label{diff1}
\ASTKE{K}{\tau} = \ASTK{K}{\tau} \setminus \ASTK{K-1}{\tau},~~~~
\forall K\in\{0,\ldots,N\},
\end{equation}
where we extend the definition of $\ASTK{K}{\tau}$ with
\[\ASTK{-1}{\tau} = \emptyset.
\]
Being the  difference of two measurable closed sets, $\ASTKE{K}{\tau}$ is clearly measurable.
Noticing also that
\begin{equation}\ASTK{K-1}{\tau}\subset\ASTK{K}{\tau} \label{K-1subsetK}\end{equation}
we get
\beq\LEBEG{ \ASTKE{K}{\tau}\cap \LS{f_d}{\theta}}{N}=\LEBEG{\ASTK{K}{\tau}\cap\LS{f_d}{\theta}}{N} -
\LEBEG{ \ASTK{K-1}{\tau}\cap \LS{f_d}{\theta}}{N}.
\label{diff}\eeq
Combining these observations with  Theorem \ref{propSSTK} yields an
important statement which is given below.

\begin{theorem}\label{corASTKE}
Let $K\in\{0,\ldots, N\}$, the norms $\|.\|$ and  $f_d$, and $(\psi_i)_{i\in I}$, be any.
Let $\theta>0$ and $\theta \geq \tau \max(\Delta_K,\Delta_{K-1})$ where
$\Delta_k$ is defined in (\ref{delmax}), for $k\in\{K-1,K\}$.
The Lebesgue measure in $\RR^N$ of the set $\ASTKE{K}{\tau}$
defined in (\ref{DTK}) satisfies
\beqn\CK{K}~ \tau^{N-K} (\theta-\hat\delta_K \tau)^K - \theta^N \eps_0'(K,\tau,\theta)
&\leq& \LEBEG{\ASTKE{K}{\tau}\cap \LS{f_d}{\theta}}{N} \\
&\leq& \CK{K}~
\tau^{N-K} (\theta+\hat\delta_K \tau)^K + \theta^N \eps_1(K,\tau,\theta),
\label{57}\eeqn
with
\begin{eqnarray*}
\eps_0'(K,\tau,\theta) & = & \eps_0(K,\tau,\theta) + \CK{K-1}
\left(\frac{\tau}{\theta}\right)^{N-(K-1)} \left(1+\hat\delta_{K-1} \frac{\tau}{\theta}\right)^{K-1},\\
\eps_1(K,\tau,\theta) & = & \eps_0(K-1,\tau,\theta) -  \CK{K-1} \left(\frac{\tau}{\theta}\right)^{N-(K-1)}
\left(1-\hat\delta_{K-1} \frac{\tau}{\theta}\right)^{K-1},
\end{eqnarray*}
where $\CK{k}$ for $k\in\{K-1,K\}$ are defined by (\ref{CK}), along with the
extension $\CK{-1} = 0$, whereas $\eps_0$ is as in Theorem \ref{propSSTK}
with the extension $\eps_0(-1,\tau,\theta)\equiv 0$.
\end{theorem}

\begin{proof}
By (\ref{diff}), we have
\beqnn \LEBEG{\ASTKE{K}{\tau}\cap \LS{f_d}{\theta}}{N}&\leq&
\mbox{Upper bound}\Big(\LEBEG{\ASTK{K}{\tau}\cap\LS{f_d}{\theta}}{N}\Big) -
\mbox{Lower bound}\Big(\LEBEG{ \ASTK{K-1}{\tau}\cap \LS{f_d}{\theta}}{N}\Big)\\
\LEBEG{\ASTKE{K}{\tau}\cap \LS{f_d}{\theta}}{N}&\geq&
\mbox{Lower bound}\Big(\LEBEG{\ASTK{K}{\tau}\cap\LS{f_d}{\theta}}{N}\Big) -
\mbox{Upper bound}\Big(\LEBEG{ \ASTK{K-1}{\tau}\cap \LS{f_d}{\theta}}{N}\Big)
\eeqnn
where the relevant upper and lower bounds were derived  in Theorem \ref{propSSTK}.
Since $\LEBEG{ \ASTK{K-1}{\tau}\cap \LS{f_d}{\theta}}{N}$ is negligible compared to
$\LEBEG{\ASTK{K}{\tau}\cap\LS{f_d}{\theta}}{N}$, the bounds corresponding to this term are
introduced in the error functions $\eps_0'(K,\tau,\theta)$ and $\eps_1(K,\tau,\theta)$.
\end{proof}

\begin{remark} Let us emphasize that Remark \ref{rmkASTK} is valid if we write
$\ASTKE{K}{\tau}$ in place of $\ASTK{K}{\tau}$. This gives the
asymptotic of the $\LEBEG{\ASTKE{K}{\tau}\cap \LS{f_d}{\theta}}{N}$ as $\frac{\tau}{\theta}$ goes to $0$.
This observation may seem surprising.
It only means that  as far as $\frac{\tau}{\theta}$ decreases, the chance
to get a solution with sparsity {\em strictly} smaller than $K$
is very small when compared to the chance of getting a sparsity $K$.
%This means that as far as $\frac{\tau}{\theta}$ decreases, the chance
%to get a solution with sparsity $K$ (for $K<N$) becomes smaller and smaller, and
%negligible at the limit $\frac{\tau}{\theta}\to 0$.
\end{remark}

\begin{remark}
In Section \ref{sparser}, we adapted Theorem \ref{propSSTK} to get
Corollary \ref{corUnifBounds}. In the latter, the gap between the
lower and upper bounds only depends on $\|.\|$, $f_d$ and $\QK{K-1}$
the latter depending on the dictionary in a controllable way.
A similar adaptation of Theorem \ref{corASTKE} is easy.
\end{remark}
\section{Statistical meaning of the results}\label{stat}

In this section we give a statistical interpretation of our main results, namely
Theorem \ref{propSSTK} and Theorem~\ref{corASTKE}.

\begin{proposition}\label{propSSTKproba}
Let $f_d$ and $\|.\|$ be any two norms and $(\psi_i)_{i\in I}$ be a dictionary in $\RR^N$.
For any $K\in\{0,\ldots, N\}$, let $\tau>0$ and $\theta$ be such that $\theta \geq \tau \Delta_K$ where
$\Delta_k$ is defined in (\ref{delmax}).
Consider a random variable $d$ with uniform distribution on $\LS{f_d}{\theta}$.
Then
\beqnn \frac{\CK{K}}{\LEBEG{\LS{f_d}{1}}{N}}
\left(\frac{\tau}{\theta}\right)^{N-K}\left(1 - \hat\delta_K \frac{\tau}{\theta}\right)^K
-\frac{\eps_0(K,\tau,\theta)}{\LEBEG{\LS{f_d}{1}}{N}}
&\leq& \PROBA{ \val(\P_d) \leq K }\\ &\leq& \
\frac{\CK{K} }{\LEBEG{\LS{f_d}{1}}{N}} \left(\frac{\tau}{\theta}\right)^{N-K}
\left(1 + \hat\delta_K \frac{\tau}{\theta}\right)^K,
\eeqnn
where
$\eps_0(K,\tau,\theta)$ is given in Theorem \ref{propSSTK}, equation \eqref{eps_0}.
Moreover we have the following asymptotical result:
\[\PROBA{ \val(\P_d) \leq K }=\frac{\CK{K}}{\LEBEG{\LS{f_d}{1}}{N}}
\left(\frac{\tau}{\theta}\right)^{N-K} + o\left(
  \left(\frac{\tau}{\theta}\right)^{N-K}
  \right)~~\mbox{as}~~\frac{\tau}{\theta}\to 0.\]
\end{proposition}

\begin{proof}
Consider the set $\ASTK{K}{\tau}$ defined by (\ref{ASTK}).
We have
\[\PROBA{ \val(\P_d) \leq K }=\PROBA{d\in \ASTK{K}{\tau}\cap \LS{f_d}{\theta}}
=\frac{\LEBEG{\ASTK{K}{\tau}\cap \LS{f_d}{\theta}}{N}}{\LEBEG{\LS{f_d}{\theta}}{N}},\]
since $d$ is uniformly distributed on $\LS{f_d}{\theta}$.
The inequality result follow from Theorem \ref{propSSTK},
equation~\eqref{R1} and uses the observation that $\LEBEG{\LS{f_d}{\theta}}{N} = \theta^N\,
\LEBEG{\LS{f_d}{1}}{N} $.

The asymptotical result is a direct consequence of
Remark \ref{rmkASTK}.
\end{proof}

\begin{remark}
Notice that, as already noticed in \eqref{CNC0}, $\CK{N} =
\LEBEG{\LS{f_d}{1}}{N}$ and the asymptotic in Proposition
\ref{propSSTKproba} reads for $K=N$
\[\PROBA{ \val(\P_d) \leq N }= 1 + o(1) ~~~~ \mbox{ as }
\frac{\tau}{\theta}\to 0.
\]
In fact a better estimate is easy to obtain in this particular
case. We know indeed that for all $d\in\RR^N$, any solution of $\P_d$
involves an independent system of elements of $(\psi_i)_{i\in I}$. (A
sparser decomposition would otherwise exist.) Therefore we know that
for all $d\in\RR^N$, $ \val(\P_d) \leq N$. This yields
\begin{equation}\label{PN=1}
\PROBA{ \val(\P_d) \leq N }= 1.
\end{equation}
\end{remark}
\begin{theorem}\label{thm=proba}
Let $f_d$ and $\|.\|$ be any two norms and $(\psi_i)_{i\in I}$ be a dictionary in $\RR^N$.
For any $K\in\{0,\ldots, N\}$, let $\tau>0$ and $\theta$ be such that $\theta \geq \tau\max(\Delta_K,\Delta_{K-1})$ where
$\Delta_k$ is defined in (\ref{delmax}).
Consider a random variable $d$ with uniform distribution on $\LS{f_d}{\theta}$.
Then we have
\beqnn \frac{\CK{K} }{\LEBEG{\LS{f_d}{1}}{N}}
\left(\frac{\tau}{\theta}\right)^{N-K} \left(1 - \hat\delta_K \frac{\tau}{\theta}\right)^K
-  \eps^-(K,\tau,\theta)
&\leq& \PROBA{ \val(\P_d) = K }\\ &\leq& \
\frac{\CK{K} }{\LEBEG{\LS{f_d}{1}}{N}}
\left(\frac{\tau}{\theta}\right)^{N-K} \left(1 + \hat\delta_K \frac{\tau}{\theta}\right)^K
+\eps^+(K,\tau,\theta)
\eeqnn
with
\[\eps^-(K,\tau,\theta) =
\frac{\eps_0'(K,\tau,\theta)}{\LEBEG{\LS{f_d}{1}}{N} }
\]
and
\[\eps^+(K,\tau,\theta) =
\frac{\eps_1(K,\tau,\theta)}{\LEBEG{\LS{f_d}{1}}{N}},
\]
for $\eps_0'$ and $\eps_1$ as defined in Theorem \ref{corASTKE} and
for $\hat\delta_K$ and $\CK{K}$ defined in (\ref{deltaK}) and (\ref{CK}), respectively.

In particular, we have
\begin{equation}\label{proba=K}
\PROBA{ \val(\P_d) = K } =
\frac{\CK{K}}{\LEBEG{\LS{f_d}{1}}{N}}
\left(\frac{\tau}{\theta}\right)^{N-K} +o
\left(\left(\frac{\tau}{\theta}\right)^{N-K}\right)~~~\mbox{as}~~~\frac{\tau}{\theta}\to
0.
\end{equation}
\end{theorem}
\begin{proof}
Consider the set $\ASTKE{K}{\tau}$ defined in \eqref{DTK}. We have
\[\PROBA{ \val(\P_d)=K}=\PROBA{d\in\ASTKE{K}{\tau}\cap \LS{f_d}{\theta}}
=\frac{\LEBEG{\ASTKE{K}{\tau}\cap \LS{f_d}{\theta}}{N}}{\LEBEG{\LS{f_d}{\theta}}{N}},\]
since $d$ is uniformly distributed on $\LS{f_d}{\theta}$.
The inequality result follows from Theorem \ref{corASTKE}, equation \eqref{57}, and $\LEBEG{\LS{f_d}{\theta}}{N} =
\LEBEG{\LS{f_d}{1}}{N} \theta^N$.
\end{proof}

\begin{remark}
>From \eqref{proba=K} and \eqref{CNC0}, we see that
\[\PROBA{ \val(\P_d) = N  } = 1 + o(1) ~~~~\mbox{ as
}\frac{\tau}{\theta} \to 0.
\]
For any other $K\in\{0,\ldots,N-1\}$, $\PROBA{ \val(\P_d) = K  }$ goes
to $0$, as $\frac{\tau}{\theta} \to 0$. Moreover, we know how rapidly
they go to $0$. In particular, we know that $\PROBA{ \val(\P_d) =
  K-1  }$ becomes negligible when compared to $\PROBA{ \val(\P_d) = K  }$, as
$\frac{\tau}{\theta} \to 0$.
\end{remark}

Notice that even though $d$  is a random variable on a subset of $\RR^N$, the
value of our function $\val(\P_d)$ is an integer larger than zero. We can also compute the
expectation of $\val(\P_d)$:
\begin{eqnarray*}
\EXPECT{\val(\P_d)} & = & \sum_{K=1}^N K ~\PROBA{\val(\P_d) = K} \\
& = & \sum_{K=1}^N K \left(\PROBA{\val(\P_d) \leq K} - \PROBA{\val(\P_d) \leq K-1}  \right) \\
& = &  \sum_{K=0}^N K ~\PROBA{\val(\P_d) \leq K} -
  \sum_{K=0}^{N-1} (K+1) \PROBA{\val(\P_d) \leq K} \\
& = & \PROBA{\val(\P_d) \leq N} - \sum_{K=0}^{N-1} \PROBA{\val(\P_d) \leq K} \\
& = & N  - \sum_{K=0}^{N-1} \PROBA{\val(\P_d) \leq K}
\end{eqnarray*}
where we used \eqref{K-1subsetK} and \eqref{PN=1}.

This yields the following Theorem.
\begin{theorem}\label{thmexpect}
Let $f_d$ and $\|.\|$ be any two norms and $(\psi_i)_{i\in I}$ be a dictionary in $\RR^N$.
Let $\tau>0$ and $\theta$ be such that $\theta \geq \tau \max_{0\leq K \leq N} \Delta_K$ where
$\Delta_K$ is defined in (\ref{delmax}). Consider a random variable
$d$ with uniform distribution on $\LS{f_d}{\theta}$. Then
\beqnn
N - \sum_{K=0}^{N-1}\frac{\CK{K} }{\LEBEG{\LS{f_d}{1}}{N}} \left(\frac{\tau}{\theta}\right)^{N-K}
\left(1 + \hat\delta_K \frac{\tau}{\theta}\right)^K&\leq& \EXPECT{ \val(\P_d) }\\ &\leq& N - \sum_{K=0}^{N-1}
\frac{\CK{K}}{\LEBEG{\LS{f_d}{1}}{N}} \left(\frac{\tau}{\theta}\right)^{N-K}\left(1 - \hat\delta_K
  \frac{\tau}{\theta}\right)^K -\frac{\eps_0(K,\tau,\theta)}{\LEBEG{\LS{f_d}{1}}{N}}
\eeqnn
where
$\eps_0(K,\tau,\theta)$ is given in Theorem \ref{propSSTK}, equation \eqref{eps_0}.
Moreover we have the following asymptotical result:
\[\EXPECT{ \val(\P_d) }= N - \frac{\CK{N-1}}{\LEBEG{\LS{f_d}{1}}{N}}
\frac{\tau}{\theta} + o\left(\frac{\tau}{\theta} \right)~~\mbox{as}~~
\frac{\tau}{\theta}\to 0.\]
\end{theorem}

% \section{A stability result}

% Consider a datum $d$ such that there exists $(\lambda_i)_{i\in I}$
% satisfying
% \[d= \sum_{i\in I} \lambda_i \psi_i.
% \]
% Consider $\tau>0$ and a noise process such that for any realization
% $b$ of the noise
% \[\|b\|\leq \tau.
% \]
% Then
% \[d + b \in \ASTK{\tau}{K}.
% \]
% That is a solution to $(P)_{d+b}$ is $K$-sparse or sparser. Notice
% however that there is no garantee that
% \begin{equation}\label{inclus}
% \{i\in I: \hat\lambda_i\neq 0\} \subset \{i\in I: \lambda_i\neq 0\}.
% \end{equation}

% Je ne vois malheureusemnent pas comment mesurer
% \[Neg = \{d\in\RR^N : \mbox{\eqref{inclus} est fausse}\}
% \]
% car on a une loi sur $b$, une loi sur $d$ et la loi sur $b+d$ sera
% plus compliqu\'ee que la loi uniforme que l'on considere. Une solution
% serait de borner la loi en question, on peut alors peut etre montrer
% que $Neg$ est negligeable.

\section{Illustration: Euclidean norms for $\|.\|$ and $f_d$}\label{illustre-sec}
Consider the situation when both $\|.\|$ and $f_d$ are the Euclidean norm on $\RR^N$:
\beq\|.\|=f_d=\|.\|_2~~~\mbox{where}~~~~\|u\|_2=\sqrt{\PS{u}{u}},~~\mbox{ with}~~~
\PS{u}{v} = \sum_{i=1}^N u_i v_i \label{Eu}.\eeq
Noticing that the Euclidean norm is rotation invariant, for any vector subspace $V\subseteq\RR^N$ we have
\beq%\LS{h}{\tau}=
P_{V^\bot}\left(\LS{\|.\|_2}{\tau}\right)=V^\bot\cap\LS{\|.\|_2}{\tau}=\{u\in V^\bot:\|u\|_2\leq\tau\}.
\label{wikiki}
\eeq
The equivalent norm $h$ and the constant $\odelta$ derived in Lemma \ref{lemme_h} are simply
\beqnn h(u)&=&\|u\|_2,~~~\forall u\in V^\bot,\\
\odelta&=&1.\eeqnn

The constant $\delta_V$ in  assertion (ii) of Proposition \ref{propV}, defined by \eqref{delV}, reads
$\delta_V=1.$
Then the inequality condition on $\theta$ and $\tau$ is simplified to $\theta\geq\tau$.

The constant $C$ in (\ref{conC}) in the same proposition depends on $K$ (the dimension of the subspace $V$) and
reads (see \cite[p.60]{Evans92} for details)
\[ C=\alpha(K)\alpha(N-K)~\defeq~\C(K),\]
where for any integer $n>0$ we have
\beq  \alpha(n)=\frac{\pi^{n/2}}{\Gamma(n/2+1)}~~~\mbox{for}~~~
\Gamma(n)=\int_0^\infty e^{-x}x^{n-1}dx. \label{alpha}
\eeq
Here $\Gamma$ is the usual Gamma function.
Using that $\Gamma(n+1)=n\Gamma(n)$, it comes
\beq \C(K)=\frac{4\pi^{\frac{N}{2}}}{K(N-K)\Gamma\big(\frac{N-K}{2}\big)\Gamma\big(\frac{K}{2}\big)} \label{Bk}\eeq
>From the preceding, the constants $\delta_J$ and $C_J$ in Corollary \ref{rmkTJ} read
\beqn  \delta_J&=&1,~~~\forall J\subset I,\label{new}\\
C_J&=& \C(K), \label{new0}
\eeqn
where the expression of $\C(K)$ is given in \eqref{Bk}.

The norm $g$ arising in (\ref{g}) in Proposition \ref{basic} reads
\beqnn g(u)&=&\sup\{\|u_1\|_2+\|u_2\|_2,~\|u_1\|_2+\|u_3\|_2\}\\
&=&\|u_1\|_2+\sup\{\|u_2\|_2,~\|u_3\|_2\}\eeqnn
where $u=u_1+u_2+u_3$ is decomposed according to (\ref{decomp}).
Then
\[f_d(u)=\|u\|_2=\|u_1\|_2+\|u_2\|_2+\|u_3\|_2 \leq \delta_{\!J_1\!,\!J_2}g(u),~~~\forall u\in W^\bot~~~\mbox{if}
~~~\delta_{\!J_1\!,\!J_2}=2\]

The constants $\delta_{\!J_1\!,\!J_2}$ and $Q_{\!J_1\!,\!J_2}$ in Proposition \ref{basic} read
\beqn  \delta_{_{\!J_1\!,\!J_2}}&=&2 \label{dmau}\\
Q_{\!J_1\!,\!J_2}&=& \C(k)\label{QJ1J2},
\eeqn
where $\C(k)$ is defined according to \eqref{Bk}.

For any $k=1,\ldots,N$, the constants $\hat\delta_k$ and $\CK{k}$ in (\ref{deltaK})-(\ref{CK}) read
\beqnn \hat\delta_k&=&1,\\
\CK{k}  &=& \C(k)~ \#\J(k).
\eeqnn
Clearly, $\#\J(K)$ depends on the dictionary $(\psi_i)_{i\in I}$.

The constants $\hat\delta'_{K,k}$ and $\QK{k}$, introduced in (\ref{dpK}) and (\ref{QK}), respectively, are
\beqn
\hat\delta'_{K,k}&=&2,\\
\QK{k}&=&~\C(k)~\#\H(K,k).
\eeqn
Here again, $\#\H(K,k)$ depends on the choice of dictionary and in any case,  $\#\H(K,k)=0$ for $k<k_0$ (where $k_0$ is
defined in \eqref{ko}). The constant in \eqref{delmax} is $\Delta_K=2$ and the inequality \eqref{Q3} is satisfied.

The main inequality in  Theorem \ref{propSSTK} now reads
\beqnn \C(K) \#\{\J(K)\}~   \tau^{N-K} (\theta- \tau)^K - \eps_0(K,\tau,\theta) &\leq&
\LEBEG{\ASTK{K}{\tau}\cap \LS{f_d}{\theta}}{N}\\
&\leq& \C(K) \#\{\J(K)\}~\tau^{N-K} (\theta+\tau)^K ,
\eeqnn
where $\C(K)$ is defined by \eqref{Bk} and
the error term $\eps_0(K,\tau,\theta)$ is
\[ \eps_0(K,\tau,\theta) =  \frac{1}{2}\sum_{k=k_0}^{K-1} \C(k)~\#\{\H(K,k)\}~\tau^{N-k} (\theta + 2 \tau)^k.\]

In order to provide the statistical interpretation in section \ref{stat}, we notice that
$\LEBEG{\LS{f_d}{1}}{N}=\alpha(N)$ for $\alpha(.)$ as given in (\ref{alpha}), and hence
\[\LEBEG{\LS{f_d}{1}}{N}=\frac{\pi^{N/2}}{\Gamma(N/2+1)}.\]

\section{Conclusion and perspectives}
In this paper, we derive lower and upper bounds for different
quantities concerning the model $(\P_d)$. Typically, the difference between the upper and
the lower bound has an order of magnitude $(\frac{\tau}{\theta})^{N-K+1}$
while the quantities which are estimated are propositional to
$(\frac{\tau}{\theta})^{N-K}$. The difference between the upper and
lower bounds is made of
\begin{itemize}
\item The terms $\theta\pm \delta_v \tau$ which come from the
  inclusions $B_0\subseteq ~V^\tau \cap \LS{f_d}{\theta}\subseteq
  B_1$, in the proof Proposition \ref{propV}. This approximation is of
  the order $(\frac{\tau}{\theta})^{N-K+1}$.
  %We could probably reach a
  %larger order of magnitude (e.g. $(\frac{\tau}{\theta})^{N-K+2}$) if
  %we assume that $f_d$ is regular away from $0$ (e.g. twice
  %differentiable).
It may be possible to reach a larger order of magnitude (e.g. $(\frac{\tau}{\theta})^{N-K+2}$)
under the assumption that $f_d$ is regular away from $0$ (e.g. twice differentiable).
  This would permit to improve Proposition
  \ref{propV} and the theorems that use its conclusions.
  %using it.
%  We recommend to take a look
%  at the proof of Proposition \ref{propV} and investigate the
%  possible adaptations, when $f_d$ is (for instance) twice
%  differentiable away from $0$.
\item A term of the form $-\theta^N \eps_0(K,\tau,\theta)$ could be added to the upper bound
  in \eqref{R1}. This term is not present because of the
  approximation made in \eqref{missing}.
Such a term ``$-\theta^N \eps_0(K,\tau,\theta)$'' could be obtained by
%  We could obtain this
%  ``$-\theta^N \eps_0(K,\tau,\theta)$'' term by
  computing
  the size of the intersection of more than two cylinder-like sets in
  Proposition \ref{basic} (doing so we would also avoid the
  approximation in \eqref{mnb}) and by improving this proposition by
  bounding $\LEBEG{\TTT{J_1}{\tau} \cap \TTT{J_2}{\tau}
    \cap \LS{f_d}{\theta}}{N}$ from below. This is probably a straightforward
  adaptation of the current proof of Proposition \ref{basic}.
  %(Other minor adaptations of
  %the notations are also probably necessary.)

  This improvement is possible but not necessary in this paper since
  (again) this approximation yields an error whose order of magnitude is
  $(\frac{\tau}{\theta})^{N-K+1}$. We can anyway not get a better
  order of magnitude unless the approximation mentioned in the
  previous item is not improved (i.e. more regularity is assumed for
  $f_d$).
\end{itemize}

Besides those aspects, several future developments can be envisaged:
\begin{itemize}
\item An important improvement would be to assume a more specialized form for
the data distribution. One first step would
  be a distribution of the shape $\propto e^{-f_d(w)}$ which is continuous. In
  our opinion, one possible goal
  is to deal with a data distribution defined by a kernel. This is indeed
  one of the standard technique used in machine learning theory to
  approximate data distributions.
\item Another way of improvement is to adapt those results to the
  context of infinite dimensional spaces. This adaptation might not be
  trivial since (for instance) there is no Lebesgue measure in those
  spaces.
\item We are also preparing a paper where a similar analysis is
  performed for the Basis Pursuit Denoising (i.e. $l^1$
  regularization) with the same asymptotic. It will clearly show what
  is in common and what are the differences between $\ell_0$ and $l^1$
  regularization.
\item Performing a similar analysis for the Orthogonal Matching
  Pursuit would, of course, be a interesting and complementary result.
\item In a forthcoming work, we develop the theory in the context of
  orthogonal bases instead of general dictionaries (frames). This
  simplification of the hypotheses simplifies a lot the formulas of the
  current paper and illustrate it.
\end{itemize}

\bibliographystyle{plain}
\markboth{}{}

% TEX 7(ascii) bits
%--------------------- month-en.tex ---------------------------
% Fichier de d\'efinition pour les mois anglais
%
\def\jan{Jan. }
\def\feb{Feb. }
\def\mar{Mar. }
\def\apr{Apr. }
\def\may{May }
\def\jun{June }
\def\jul{July }
\def\aug{Aug. }
\def\sep{Sep. }
\def\oct{Oct. }
\def\nov{Nov. }
\def\dec{Dec. }
\def\Jan{Jan. }
\def\Feb{Feb. }
\def\Mar{Mar. }
\def\Apr{Apr. }
\def\May{May }
\def\Jun{June }
\def\Jul{July }
\def\Aug{Aug. }
\def\Sep{Sep. }
\def\Oct{Oct. }
\def\Nov{Nov. }
\def\Dec{Dec. }
% plus
\def\sub{submitted to }

% TEX 7(ascii) bits
% fichier : revuedef.tex    Auteur: AMD, GLB    09/92
% Modifie et completee par  L.Lemitre le 05 10 94
%               et al...
%
% Ce fichier contient des alias sur les noms de revues
% pour faciliter la saisie de fiches biblio.
%
% En commentaire, le facteur d'impact (nombre moyen de citation
% des articles, normalise a la moyenne mondiale sur l'ensemble
% des revues indexees par ISI) pour l'annee 1994 (J. I.).

\def\AsAs{Astrononmy and Astrophysics}                  % 2.328
\def\AAP{Advances in Applied Probability}               % 0.549
\def\ABE{Annals of Biomedical Engineering}              % 0.709
\def\ACFAS{Proceedings of ACFAS Conference}             % ?
\def\AISM{Annals of Institute of Statistical Mathematics}       % 0.175
\def\AMS{Annals of Mathematical Statistics}         % ANNALS OF MATHEMATICS STUDIES ? 0.8
\def\AO{Applied Optics}                         % 1.033
\def\AP{The Annals of Probability}                  % 0.68
\def\ARAA{Annual Review of Astronomy and Astrophysics}          % 9.865
\def\AST{The Annals of Statistics}                  % 0.78
\def\AT{Annales des T\'el\'ecommunications}             % 0.155
\def\BMC{Biometrics}                            % 1.207
\def\BMK{Biometrika}                            % 0.832
\def\CPAM{Communications on Pure and Applied Mathematics}       % 1.282
\def\EMK{Econometrica}                          % 2.54
\def\CRAS{Compte-rendus de l'acad\'emie des sciences}           % ?
\def\CVGIP{Computer Vision and Graphics and Image Processing}       % ?
\def\GJRAS{Geophysical Journal of the Royal Astrononomical Society} % 1.371
\def\GSC{Geoscience}                        % GEOSCIENCE CANADA ? 0.273
\def\GPH{Geophysics}                            % 0.824
\def\GRETSI#1{Actes du #1$^{\mbox{e}}$ Colloque GRETSI}         % ?
\def\CGIP{Computer Graphics and Image Processing}           % ?
\def\ICASSP{Proceedings of the IEEE Int. Conf. on Acoustics,
    Speech and Signal Processing}                   % ?
\def\ICEMBS{Proceedings of IEEE EMBS}                   % ?
\def\ICIP{Proceedings of the IEEE International Conference on Image Processing}% ?
\def\EUSIPCO{Proceedings of European Signal Processing Conference} % ?
\def\SSAP{Proceedings of the IEEE Statistical Signal and Array Processing} % ?
\def\ieeP{Proceedings of the IEE}                   % ?
\def\ieeeAC{IEEE Transactions on Automatic and Control}         % 0.867
\def\ieeeAES{IEEE Transactions on Aerospace and Electronic Systems} % 0.459
\def\ieeeAP{IEEE Transactions on Antennas and Propagation}      % 0.806
\def\ieeeAPm{IEEE Antennas and Propagation Magazine}            % ?
\def\ieeeASSP{IEEE Transactions on Acoustics Speech and Signal Processing}% ?
\def\ieeeBME{IEEE Transactions on Biomedical Engineering}       % 1.061
\def\ieeeCS{IEEE Transactions on Circuits and Systems}          % 0.54
\def\ieeeCT{IEEE Transactions on Circuit Theory}            % 0.732
\def\ieeeC{IEEE Transactions on Communications}             % 0.969
\def\ieeeGE{IEEE Transactions on Geoscience and Remote Sensing}     % 1.356
\def\ieeeGEE{IEEE Transactions on Geosciences Electronics}      % ?
\def\ieeeIP{IEEE Transactions on Image Processing}          % ?
\def\ieeeIT{IEEE Transactions on Information Theory}            % 1.971
\def\ieeeMI{IEEE Transactions on Medical Imaging}           % 1.372
\def\ieeeMTT{IEEE Transactions on Microwave Theory and Technology}  % 1.004
\def\ieeeM{IEEE Transactions on Magnetics}              % 0.758
\def\ieeeNS{IEEE Transactions on Nuclear Sciences}          % 1.183
\def\ieeePAMI{IEEE Transactions on Pattern Analysis and Machine Intelligence}% 2.006
\def\ieeeP{Proceedings of the IEEE}                 % ?
\def\ieeeRS{IEEE Transactions on Radio Science}             % ?
\def\ieeeSMC{IEEE Transactions on Systems, Man and Cybernetics}     % 0.649
\def\ieeeSP{IEEE Transactions on Signal Processing}         % 1.234
\def\ieeeSSC{IEEE Transactions on Systems Science and Cybernetics}  % ?
\def\ieeeSU{IEEE Transactions on Sonics and Ultrasonics}        % ?
\def\ieeeUFFC{IEEE Transactions on Ultrasonics Ferroelectrics and Frequency Control}% 0.927
\def\IJC{International Journal of Control}              % 0.617
\def\IJCV{International Journal of Computer Vision}         % 1.153
\def\IJIST{International Journal of Imaging Systems and Technology} % ?
\def\IP{Inverse Problems}                       % 0.98
\def\ISR{International Statistical Review}              % 0.75
\def\IUSS{Proceedings of International Ultrasonics Symposium}       % ?
\def\JAPH{Journal of Applied Physics}                   % 1.658
\def\JAP{Journal of Applied Probability}                % 0.432
\def\JAS{Journal of Applied Statistics}                 % ?
\def\JASA{Journal of Acoustical Society America}            % 1.273
\def\JASAS{Journal of American Statistical Association}         % 1.244
\def\JBME{Journal of Biomedical Engineering}                % 0.7
\def\JCAM{Journal of Computational and Applied Mathematics}     % 0.349
\def\JCP{Journal of Computational Physics}
\def\JEWA{Journal of Electromagnetic Waves and Applications}        % 0.514
\def\JMIV{Journal of Mathematical Imaging and Vision}
\def\JMO{Journal of Modern Optics}                  % 1.005
\def\JNDE{Journal of Nondestructive Evaluation}             % ?
\def\JMP{Journal of Mathematical Physics}               % 0.969
\def\JOSA{Journal of the Optical Society of America}            % 1.425 (A-OPTICS IMAGE SCIENCE AND VISION)
\def\JP{Journal de Physique}                        % 1.773 JOURNAL DE PHYSIQUE I
                                    % 1.923 JOURNAL DE PHYSIQUE II
                                    % 0.475 JOURNAL DE PHYSIQUE III
                                    % 0.115 JOURNAL DE PHYSIQUE IV
\def\JRSSA{Journal of the Royal Statistical Society A}          % 0.976
\def\JRSSB{Journal of the Royal Statistical Society B}          % 2.538
\def\JRSSC{Journal of the Royal Statistical Society C}          % ?
\def\JSPI{Journal of Statistical Planning and Inference}        % 0.145
\def\JTSA{Journal of Time Series Analysis}                          % ?
\def\JVCIR{Journal of Visual Communication and Image Representation}    % ?
\def\KAP{Kluwer \uppercase{A}cademic \uppercase{P}ublishers}                            %
\def\MNAS{Mathematical Methods in Applied Science}          % 0.247
\def\MNRAS{Monthly Notices of the Royal Astronomical Society}       % 3.089
\def\MP{Mathematical Programming}                   % 0.763
    \def\NSIP{NSIP}  % Trouve dans gpi base             % ?
\def\OC{Optics Communication}                       % 1.205
\def\PRA{Physical Review A}                     % 2.292
\def\PRB{Physical Review B}                     % 3.187
\def\PRC{Physical Review C}                     % 1.842
\def\PRD{Physical Review D}                     % 3.233
\def\PRL{Physical Review Letters}                   % 6.626
\def\RGSP{Review of Geophysics and Space Physics}           % 4.314
\def\RPA{Revue de Physique Appliqu\'e}                          % 0.753
\def\RS{Radio Science}                          % 0.753
\def\SP{Signal Processing}                      % 0.44
\def\siamAM{SIAM Journal on Applied Mathematics}            % 0.743
\def\siamCO{SIAM Journal on Control and Optimization}                   % 0.968
\def\siamC{SIAM Journal on Control}                 % 0.968
\def\siamJO{SIAM Journal on Optimization}               % ?
\def\siamMA{SIAM Journal on Mathematical Analysis}          % 0.765
\def\siamNA{SIAM Journal on Numerical Analysis}             % 1.021
\def\siamSC{SIAM Journal on Scientific Computing}             % 1.021
\def\siamMMS{SIAM Journal on Multiscale Modeling and Simulation}             %
\def\siamMatrix{SIAM Journal on Matrix Analysis and Applications}
\def\siamR{SIAM Review}                         % 1.216
\def\SSR{Stochastics and Stochastics Reports}                   % ?
\def\TPA{Theory of Probability and its Applications}            % 0.045
\def\TMK{Technometrics}                         % 1.585
\def\TS{Traitement du Signal}                       % ?
\def\UCMMP{U.S.S.R. Computational Mathematics and Mathematical Physics} % ?
\def\UMB{Ultrasound in Medecine and Biology}                % 1.252
\def\US{Ultrasonics}                            % 1.046
\def\USI{Ultrasonic Imaging}                        % 1.409

\bibliography{sparsest.bbl}

\begin{thebibliography}{10}

\bibitem{Besag86}
Julian~E. Besag.
\newblock On the statistical analysis of dirty pictures (with discussion).
\newblock {\em \JRSSB}, 48(3):259--302, 1986.

\bibitem{Besag89}
Julian~E. Besag.
\newblock Digital image processing : Towards \uppercase{B}ayesian image
  analysis.
\newblock {\em \JAS}, 16(3):395--407, 1989.

\bibitem{ChenDonoho}
S.~S. Chen, D.~L. Donoho, and M.~A. Saunders.
\newblock Atomic decomposition by basis pursuit.
\newblock {\em SIAM Journal on Scientific Computing}, 20(1):33--61, 1999.

\bibitem{CoifmanWickerhauser}
R.R. Coifman and M.V. Wickerhauser.
\newblock Entropy-based algorithms for best basis selection.
\newblock {\em IEEE, Transactions on Information Theory}, 38(2):713--718, March
  1992.

\bibitem{Devore98}
R.A. Devore.
\newblock Nonlinear approximation.
\newblock {\em Acta Numerica}, 7:51--150, 1998.

\bibitem{Donoho92}
David Donoho, Iain Johnstone, Jeffrey Hoch, and Alan Stern.
\newblock Maximum entropy and the nearly black object.
\newblock {\em \JRSSB}, 54(1):41--81, 1992.

\bibitem{Evans92}
L.~C. Evans and R.~F. Gariepy.
\newblock {\em Measure theory and fine properties of functions}.
\newblock Studies in Advanced Mathematics. CRC Press, Roca Baton, FL, 1992.

\bibitem{Geman84}
Stuart Geman and Donald Geman.
\newblock Stochastic relaxation, \uppercase{G}ibbs distributions, and the
  \uppercase{B}ayesian restoration of images.
\newblock {\em \ieeePAMI}, PAMI-6(6):721--741, \Nov 1984.

\bibitem{GormishLeeMarcellinJPEG2000}
M.J. Gormish, D.~Lee, and M.W. Marcellin.
\newblock Jpeg 2000: overview architecture and applications.
\newblock In {\em proc. of ICIP 2000}, volume~2, 2000.

\bibitem{Leclerc89}
Y.G. Leclerc.
\newblock Constructing simple stable description for image partitioning.
\newblock {\em \IJCV}, 3:73--102, 1989.

\bibitem{Li95}
S.Z. Li.
\newblock {\em Markov Random Field Modeling in Computer Vision}.
\newblock Springer-Verlag, New York, 1 edition, 1995.

\bibitem{Luenberger69}
D.G. Luenberger.
\newblock {\em Optimization by Vector Space Methods}.
\newblock Wiley, J., New York, 1 edition, 1969.

\bibitem{Malgouyres-compres-estim-proof}
F.~Malgouyres.
\newblock Estimating the probability law of the codelength as a function of the
  approximation error in image compression.
\newblock {\em Comptes Rendus de l'Académie des sciences, série mathématiques},
  344(9):607--610, 2007.

\bibitem{MalgouyresCompression}
F.~Malgouyres.
\newblock Image compression through a projection onto a polyhedral set.
\newblock {\em Journal of Mathematical Imaging and Vision}, 27(2):193--200,
  Feb. 2007.

\bibitem{MalgouyresSigPro}
F.~Malgouyres.
\newblock Rank related properties for basis pursuit and total variation
  regularization.
\newblock {\em Signal Processing}, 87(11):2695--2707, Nov. 2007.

\bibitem{Malgouyres06a}
François Malgouyres.
\newblock Projecting onto a polytope simplifies data distributions.
\newblock Technical report, University Paris 13, 2006-1, Jaunuary, 2006.

\bibitem{MallatZhang}
S.~Mallat and Z.~Zhang.
\newblock Matching pursuits with time-frequency dictionaries.
\newblock {\em IEEE, Transactions on Signal Processing}, 41(12):3397--3415,
  December 1993.

\bibitem{pati93OMP}
Y.~Pati, R.~Rezaiifar, and P.~Krishnaprasad.
\newblock Orthogonal matching pursuit: Recursive function approximation with
  applications to wavelet decomposition.
\newblock In {\em 27 th Annual Asilomar Conference on Signals, Systems, and
  Computers}, volume~1, pages 40--44. IEEE, 93.

\bibitem{Robini07}
M.C. Robini, A.~Lachal, and I.E. Magnin.
\newblock A stochastic continuation approach to piecewise constant
  reconstruction.
\newblock {\em \ieeeIP}, 16(10):2576--2589, \Oct 2007.

\bibitem{WallaceJPEG}
G.K. Wallace.
\newblock The jpeg still picture compression standard.
\newblock {\em Communications of the ACM}, 34(4):30--44, April 1991.

\end{thebibliography}

%\bibliography{mesRef}
\end{document}